\documentclass[oneside,english,a4wide]{amsart}
\usepackage[latin9]{inputenc}
\usepackage[a4paper]{geometry}
\usepackage{mathabx}
\usepackage{esvect}
\usepackage{babel}
\usepackage{mathtools}
\usepackage{amstext}
\usepackage{amsthm}
\usepackage{amssymb}
\usepackage{enumerate}
\usepackage{esint}
\usepackage{graphicx}
\usepackage{bbm}
\usepackage[all]{xy}
\usepackage{mathrsfs}
\usepackage[unicode=true,
bookmarks=true,bookmarksnumbered=true,bookmarksopen=true,bookmarksopenlevel=2,
breaklinks=false,pdfborder={0 0 1},backref=false,colorlinks=false]
{hyperref}
\hypersetup{colorlinks,linkcolor=blue,anchorcolor=blue,citecolor=blue}

\usepackage[svgnames]{xcolor} % Required for colour specification

\usepackage{graphicx} % Required for box manipulation
\usepackage{cite} % Allow citation

\makeatletter
\numberwithin{equation}{section}
\numberwithin{figure}{section}
\theoremstyle{plain}
\newtheorem{thm}{\protect\theoremname}[section]
\theoremstyle{plain}
\newtheorem{prop}[thm]{\protect\propositionname}

\newtheorem{corollary}[thm]{Corollary}
\theoremstyle{definition}

\theoremstyle{plain}

\newtheorem{lemma}[thm]{Lemma}
\theoremstyle{remark}

\theoremstyle{definition}
\newtheorem{definition}[thm]{Definition}
\newtheorem*{claim*}{Claim}

\theoremstyle{remark}
\newtheorem{remark}[thm]{Remark}
\theoremstyle{definition}
\newtheorem*{defn*}{\protect\definitionname}

\usepackage{babel}

\makeatother

\providecommand{\definitionname}{Definition}
\providecommand{\lemmaname}{Lemma}
\providecommand{\propositionname}{Proposition}
\providecommand{\remarkname}{Remark}
\providecommand{\theoremname}{Theorem}

\newcommand{\Rmnum}[1]{\uppercase\expandafter{\romannumeral #1}}
\makeatother

\newcommand{\real}{{\mathbb R}}
\newcommand{\nat}{{\mathbb N}}
\newcommand{\ent}{{\mathbb Z}}
\newcommand{\com}{{\mathbb C}}
\newcommand{\ce}{{\mathbb E}}

\newcommand{\A}{{\mathcal A}}

\newcommand{\D}{{\mathscr{D}}}

\newcommand{\F}{{\mathcal F}}
\renewcommand{\L}{{\mathcal L}}
\newcommand{\M}{{\mathcal M}}
\newcommand{\N}{{\mathcal N}}

\newcommand{\R}{{\mathcal R}}
\renewcommand{\S}{\mathcal{S}}

\newcommand{\Q}{{\mathcal Q}}

\newcommand{\e}{\varepsilon}
\newcommand{\f}{\varphi}

\newcommand{\ot}{\otimes}

\renewcommand{\1}{\mathbbm{1}}

\newcommand{\wt}{\widetilde}
\newcommand{\wh}{\widehat}

\newcommand{\be}{\begin{eqnarray*}}
	\newcommand{\ee}{\end{eqnarray*}}
\newcommand{\beq}{\begin{equation}}
\newcommand{\eeq}{\end{equation}}
\newcommand{\beqn}{\begin{equation*}}
\newcommand{\eeqn}{\end{equation*}}
\newcommand{\bs}{\begin{split}}
	\newcommand{\es}{\end{split}}

\newcommand{\K}{{\mathcal K}}

\newcommand\norm[1]{ \left\| #1 \right\| }
\newcommand\jdz[1]{ \left| #1 \right| }
\newcommand\sk[1]{ \left( #1 \right) }
\newcommand\mk[1]{ \left[ #1 \right] }
\newcommand\lk[1]{ \left\{ #1 \right\} }
\newcommand{\lara}[1]{\left\langle #1 \right\rangle}

\newcommand\supp{\text{supp}}

\renewcommand\d{{\rm d}}
\newcommand\Arg{{\rm Arg}}
\newcommand{\at}{{\rm at}}

\begin{document}
\pagestyle{myheadings}
\markboth{}{Best constants in the vector-valued L-P-S theory}	
%	\title{}
%	\author{}
	
%	\maketitle
%	\tableofcontents{}
%	\pagestyle{headings}
	
%	\newpage
\title{Best constants in the vector-valued Littlewood-Paley-Stein theory}


\author[Guixiang Hong]{Guixiang Hong}
\address{Institute for Advanced Study in Mathematics, Harbin Institute of Technology,  Harbin 150001, China}
\email{gxhong@hit.edu.cn}

\author[Zhendong  Xu]{Zhendong Xu}
\address{%Institute for Advanced Study in Mathematics, Harbin Institute of Technology,  Harbin 150001, China; and
	Laboratoire de Math{\'e}matiques, Universit{\'e} de Bourgogne Franche-Comt{\'e}, 25030 Besan\c{c}on Cedex, France}
\email{xu.zhendong@univ-fcomte.fr}

\author[Hao Zhang]{Hao Zhang}
\address{%Institute for Advanced Study in Mathematics, Harbin Institute of Technology,  Harbin 150001, China; and
Department of Mathematics, University of Illinois Urbana-Champaign, USA}
\email{hzhang06@illinois.edu}
\date{}

\begin{abstract}
Let $L$ be a sectorial operator of type $\alpha$ ($0 \leq \alpha < \pi/2$) on $L^2(\real^d)$ with the kernels of $\{e^{-tL}\}_{t>0}$ satisfying certain size and regularity conditions. Define
$$ S_{q,L}(f)(x) = \sk{\int_0^{\infty}\int_{|y-x| < t} \norm{tL{e^{-tL}} (f)(y) }_X^q \,\frac{\d y\d t}{t^{d+1}}}^{\frac{1}{q}},$$
$$G_{q,{L}}(f)=\sk{\int_0^{\infty} \norm{t{L}{e^{-t{L}}} (f)(y) }_X^q \,\frac{\d t}{t}}^{\frac{1}{q}}.$$
We show that for $\underline{\mathrm{any}}$  Banach space $X$, $1 \leq p < \infty$ and $1 < q < \infty$  and $f\in C_c(\mathbb R^d)\otimes X$, there hold
\begin{align*}
    p^{-\frac{1}{q}}\| S_{q,{\sqrt{\Delta}}}(f) \|_p \lesssim_{d, \gamma, \beta} \norm{ S_{q,L}(f) }_p \lesssim_{d, \gamma, \beta} p^{\frac{1}{q}}\| S_{q,{\sqrt{\Delta}}}(f) \|_p, 
    \end{align*}
        \begin{align*}
    p^{-\frac{1}{q}}\| S_{q,L}(f) \|_p \lesssim_{d, \gamma, \beta} \| G_{q,L}(f) \|_p \lesssim_{d, \gamma, \beta} p^{\frac{1}{q}}\| S_{q,L}(f) \|_p,
\end{align*}
where $\Delta$ is the standard Laplacian; moreover all the orders appeared above are {\it optimal} as $p\rightarrow1$. This, combined with the existing results in \cite{MTX,Ouyang2010}, allows us to resolve partially Problem 1.8, Problem A.1 and Conjecture A.4 regarding the optimal Lusin type constant and the characterization of martingale type in a recent remarkable work due to Xu \cite{Xu2021} . 

Several difficulties originate from the arbitrariness of $X$, which excludes the use of vector-valued Calder\'on-Zygmund theory. To surmount the obstacles,  we introduce the novel vector-valued Hardy and BMO spaces associated with sectorial operators;  in addition to Mei's duality techniques and Wilson's intrinsic square functions developed in this setting, the key new input is the vector-valued tent space theory and its unexpected amalgamation with these `old' techniques. 
\end{abstract}

\maketitle

\section{Introduction}\label{Introduction}
{Motivated by Banach space geometry \cite{Pi1975,Pi1986} and Stein's semigroup theory \cite{Stein1970}, the investigation of the vector-valued Littlewood-Paley-Stein theory has started with Xu's Poisson semigroup on the unit circle \cite{Xu1998}, and then was continued in \cite{MTX, Xu2015 ,Xu2020} for symmetric Markovian semigroups. Afterwards, Betancor {\it et al} developed this theory in some special cases which are not Markovian (cf. \cite{Betancor2007, Betancor2011, Betancor2011area,betancor2013}), such as Schr\"odinger, Hermite, Laguerre semigroups {\it etc.}, see also \cite{Abu2011,  Betancor2016, Betancor2020, Hytonen2007,Ouyang2010, Torrea2014, Harboure2003} for related results}. 
In a recent remarkable paper \cite{Xu2021}, Xu investigated for the first time the vector-valued Littlewood-Paley-Stein inequalities for semigroups of regular contractions $\{e^{-tL}\}_{t>0}$ on $L^p(\Omega)$ for a {\it fixed} $1 < p < \infty$. That is, for a Banach space $X$ of {martingale cotype} $q$ ($2 \leq q < \infty$), he showed the Lusin cotype of $X$ relative to $\{e^{-t\sqrt{L}}\}_{t>0}$, in other words, there exists a constant $C>0$ such that 
\begin{align}\label{cotype}
\|G_{q,\sqrt{L}}(f)\|_{p}\leq C\|f\|_{L^p(X)},\quad\forall f\in L^p(\Omega)\otimes X,
\end{align}
where 
$$G_{q,\sqrt{L}}(f)(x)=\sk{\int_0^{\infty} \norm{t\sqrt{L}{e^{-t\sqrt{L}}} (f)(x) }_X^q \,\frac{\d t}{t}}^{\frac{1}{q}}.$$
More importantly, by deeply exploring holomorphic functional calculus, Fendler's dilation, Calder\'on-Zygmund theory and Wilson's intrinsic square functions, he was able to obtain the sharp bounds depending on the martingale cotype constant, and the latter in turn enables him to resolve an open problem posed by Naor and Young \cite{Naor2022}. More precisely, let $\mathsf{L}^{\sqrt{L}}_{c,q, p}(X)$ be the least constant $C$ in \eqref{cotype}---the Lusin cotype constant of $X$, and $\mathsf{M}_{c,q}(X)$ the martingale cotype $q$ constant of $X$, he obtained
\begin{align}\label{cotypeconstant}  \mathsf{L}^{\sqrt{L}}_{c,q, p}(X) \lesssim \max\lk{ p^{\frac{1}{q}}, p' }\mathsf{M}_{c, q}(X)\end{align}
with the order $\max\lk{ p^{\frac{1}{q}}, p' }$ being sharp. We refer the reader to Section \ref{Applications} for the definition of $\mathsf{M}_{c, q}(X)$ and the martingale type constant $\mathsf{M}_{t, q}(X)$. 

By duality, the converse inequality of \eqref{cotype} also holds under the condition that $X$ is of  {martingale type} $q$ ($1<q \leq2$)
$$\|f-\mathbf{F}(f) \|_{L^p(X)}\leq C\|G_{q,\sqrt{L}}(f)\|_{p},\quad\forall f\in L^p(\Omega)\otimes X,$$
where $ \mathbf{F} $ is the obvious vector-valued extension of the projection from $L^p(\Omega)$ onto the fixed point space of $\{e^{-tL}\}_{t > 0}$, and the resulting {\it type} bounds satisfy
\begin{align}\label{typeconstant}
\mathsf{L}^{\sqrt{L}}_{t,q, p}(X) \lesssim \max\lk{ p, p'^{\frac{1}{q'}} }\mathsf{M}_{t, q}(X).
\end{align}
Nevertheless the order $\max\lk{ p, p'^{\frac{1}{q'}}}$ is now very likely to be suboptimal suggested by the special case $L=\Delta$---the Laplacian on $\mathbb R^d$, $q=2$ and $X=\mathbb C$,  where 
\begin{align}\label{type LC2}\mathsf{M}_{t, 2}(\mathbb C)=1, \ \mathrm{and}\ \sqrt{p}\lesssim \mathsf{L}^{\sqrt{\Delta}}_{t,2, p}(\mathbb C) \lesssim p, \end{align}
see for instance \cite[Theorem 1]{Xu2022}. The sharpness of \eqref{type LC2} when $p\rightarrow 1$ is essentially equivalent to the fact that $L^1(\mathbb R^d)$-norm of the classical $g$-function controls that of the Lusin square function, which dominates in turn $L^1(\mathbb R^d)$-norm of the function itself; this involves the deep theory of Hardy/BMO spaces. Other than this special case, the problem of determining the optimal order of $\mathsf{L}^{\sqrt{L}}_{t, q, p}(X)$ when $p\rightarrow1$ in \eqref{typeconstant} has been left open widely even in the case $L=\Delta$, see e.g. Remark 1.3, Problem 1.8 and Problem 8.4 in the aforementioned paper \cite{Xu2021}. For the other endpoint-side, the optimal order of $\mathsf{L}^{\sqrt{L}}_{t,q, p}(X)$ as $p\rightarrow\infty$ has been determined in \cite{XZ2023} for all symmetric Markovian semigroups. However it seems much harder to consider the corresponding problem for a fixed semigroup, and actually the special case $\mathsf{L}^{\sqrt{\Delta}}_{t,2, p}(\mathbb C)$ remains open {(cf. \cite[Problem 5]{Xu2022})}. 

In the present paper, we will determine the optimal order of $\mathsf{L}^{{L}}_{t, q, p}(X)$ as $p\rightarrow1$ in \eqref{typeconstant} for a large class of approximation identities $\{e^{-tL}\}_{t>0}$ on $\mathbb R^d$, and thus answer the questions mentioned in \cite[Remark 1.3 and Problem 1.8]{Xu2021}. Moreover,  our result will assert that the Lusin type of $X$ relative to this class of approximation identities implies the martingale type of $X$, and thus partially resolves \cite[Problem A.1 and Conjecture A.4]{Xu2021}.

Let $L$ be a sectorial operator of type $\alpha$ ($0 \leq \alpha < \pi/2$) on $L^2(\real^d)$, and thus it generates a holomorphic semigroup  $e^{-zL}  $ with $0\leq|\Arg(z)|<\pi/2-\alpha$. Partially inspired by \cite[Section 6.2.2]{DXT}, the kind of approximation identity $\{e^{-tL}\}_{t>0}$ that we will be interested in in the present paper is assumed to have kernel $K(t, x,y)$ satisfying the following three assumptions:  there exist positive constants $ 0 <\beta, \gamma \leq 1  $ and $c$ such that for any $t>0$, $ x,y,h \in \real^d $,

\begin{equation}\label{pt10}
    |K(t, x, y)| \leq \frac{c t^\beta }{(t + |x - y|)^{d + \beta}},
\end{equation}
\begin{equation}\label{pt2}
  | K(t, x+h, y) - K(t,  x, y) | + | K(t, x, y+h) - K(t, x, y) | \leq \frac{c|h|^\gamma t^\beta}{(t + |x-y|)^{d+\beta+\gamma}}
\end{equation}
whenever $ 2|h| \leq t + |x-y| $, and 
\begin{equation}\label{pt3}
   \int_{\real^d} K(t, x, y) \,\d x = \int_{\real^d} K(t, x, y)\,\d y = 1.
\end{equation}
One may find these concepts in Section \ref{Preliminaries}. Then it is well-known (see e.g. \cite{Xu2021}) that the semigroup $\{e^{-tL}\}_{t>0}$ extends to $L^p(\real^d; X)$ ($1\leq p \leq \infty$), where $L^p(\real^d; X)$ is the space of all strongly measurable functions $f: \real^d \to X$ such that $\|f(x)\|_X \in L^p(\real^d)$. The resulting semigroup is still denoted by $\{e^{-tL}\}_{t>0}$ without confusion.

Let $1<q<\infty$, the $q$-variant of Lusin area integral %and Littlewood-Paley $g$-function 
associated with $L$ is defined as follows: for $f\in C_c(\mathbb R^d)\otimes X$,
\begin{gather*}
    S_{q,L}(f)(x) = \sk{\int_0^{\infty}\int_{|y-x| < t} \norm{tL{e^{-tL}} (f)(y) }_X^q \,\frac{\d y\d t}{t^{d+1}}}^{\frac{1}{q}}.
    %G_{q,L}(f)(x) = \sk{\int_0^{\infty} \norm{tLe^{-tL} (f)(y) }_X^q \,\frac{\d t}{t}}^{\frac{1}{q}}.
\end{gather*}

Our main result reads as below.

\begin{thm}\label{main2}
Let $L$ be a sectorial operator of type $\alpha$ ($0 \leq \alpha < \pi/2$) on $L^2(\real^d)$ satisfying \eqref{pt10}, \eqref{pt2} and \eqref{pt3}. Let $1 \leq p < \infty$ and $1 < q < \infty$. For ${\underline {any}}$ Banach space $X$ and $f\in C_c(\mathbb R^d)\otimes X$, there hold
\begin{align}\label{SL=SD}
    p^{-\frac{1}{q}}\| S_{q,{\sqrt{\Delta}}}(f) \|_p \lesssim_{\gamma, \beta} \norm{ S_{q,L}(f) }_p \lesssim_{\gamma, \beta} p^{\frac{1}{q}}\| S_{q,{\sqrt{\Delta}}}(f) \|_p, 
    \end{align}
        \begin{align}\label{SL=SG}
    p^{-\frac{1}{q}}\| S_{q,L}(f) \|_p \lesssim_{\gamma, \beta} \| G_{q,L}(f) \|_p \lesssim_{\gamma, \beta} p^{\frac{1}{q}}\| S_{q,L}(f) \|_p. 
\end{align}
{Moreover, the orders in both \eqref{SL=SD} and \eqref{SL=SG} are {\rm optimal} as $p\rightarrow1$.}
\end{thm}

When $X=\mathbb C$ and $q=2$, the equivalence \eqref{SL=SD} in the case $1<p<\infty$ without explicit orders follows from the classical Littlewood-Paley theory which in turn relies on Calder\'on-Zygmund theory; while the case $p=1$ is deduced from the holomorphic functional calculus, Calder\'on-Zygmund theory and the theory of Hardy and BMO spaces associated with differential operators (cf. \cite[Theorem 6.10]{DXT}). Our estimate \eqref{SL=SD} for any Banach space $X$, any $1 \leq p < \infty$ and any $1 < q < \infty$  goes much beyond this and its proof provides a new approach to the mentioned scalar case with optimal orders as $p\rightarrow1$. Indeed, the arbitrariness of $X$ presents a surprise and usually one expects certain property of Banach space geometry to be imposed on the square function inequalities. For the technical side, the arbitrariness of $X$ prevents us from the use of (vector-valued) Calder\'on-Zygmund theory. Instead, we will make use of vector-valued Wilson's intrinsic square functions as a media to relate $\Delta$ and $L$, and then exploit the vector-valued tent space theory such as interpolation, duality as well as atomic decomposition. Even though both of these two tools have been developed or applied in the literature, they need to be taken care of in the present setting. For instance, because our $L$'s are usually not translation invariant or of scaling structure, we have to introduce Wilson's intrinsic square functions via nice functions of two variables satisfying \eqref{Nicef1}, \eqref{Nicef2} and \eqref{Nicef3}; to avoid the use of Calder\'on-Zygmund theory to deal with Wilson's intrinsic square functions (cf. \cite{Wilson2007,Xu2022}), we prove the boundedness of a linear operator $\mathcal K$ on vector-valued tent spaces (see Lemma \ref{BoundTent}); last but not the least, since our interested $X$ is arbitrary, one cannot  establish the basic theory of vector-valued tent space using Calder\'on-Zygmund theory as in \cite{Hytonen2008, Kemppainen2014,Kemppainen2016, Harboure1991}, and we shall adapt {the classical arguments} (cf. \cite{Coifman1985}), see Section \ref{Tentspace} for details. 

After all the preparing work, the equivalence \eqref{SL=SD} will be an immediate consequence of Theorem \ref{equiv}, where we collect all the intermediate estimates involving vector-valued Wilson's square functions. 

Regarding another equivalence \eqref{SL=SG}, in the special situation $X=\mathbb C$ and $q=2$ and $L=\sqrt{\Delta}$, the equivalence for $1<p<\infty$ without optimal orders comes from the classical Littlewood-Paley theory while the case $p=1$ constitutes one essential part of the famous real variable theory on Hardy spaces (cf. {\cite{Fefferman1971BMO, Fefferman1971,Fefferman1972hp} }); in particular the upper estimate of \eqref{SL=SG} follows from harmonicity of Poisson integrals or Calder\'on-Zygmund theory. Again, the arbitrariness of $X$ excludes the use of vector-valued Calder\'on-Zygmund theory and there is an obvious lack of harmonicity related to general $L$. To surmount these difficulties, in addition to the application of Theorem \ref{equiv}---Wilson's intrinsic square functions, we will fully develop the duality theory between vector-valued Hardy and BMO type spaces in Section \ref{Proof}; the latter is inspired by Mei's duality arguments \cite{Mei2007} (see also \cite{Xia2016,Xu2022}). In turn, part of the theory of vector-valued Hardy and BMO spaces will be deduced from vector-valued tent spaces, and the projection $\pi_L$ (see Lemma \ref{l3.8}) will play a key role in passing from the results about tent spaces to those on Hardy/BMO spaces.

%\begin{corollary}\label{cor1}
%Let $L$ be a sectorial operator of type $\alpha$ ($0 \leq \alpha < \pi/2$) on $L^2(\real^d)$ satisfying \eqref{pt10}, \eqref{pt2} and \eqref{pt2}. If a Banach space $X$ is of Lusin type (resp. cotype) $q$ relative to $\{e^{-tL}\}_{t>0}$, then $X$ is of martingale type (resp. cotype) $q$. 
%\end{corollary}

Together with the related results in \cite{MTX, Ouyang2010} where the authors showed the Lusin {type} $q$ of a Banach space $X$ relative to $\{e^{-t\sqrt{\Delta}}\}_{t>0}$ is equivalent to the martingale type $q$ of $X$ {(see Section \ref{Applications})}, our vector-valued tent space theory and Theorem \ref{main2} imply the following result, resolving partially \cite[Problem 1.8, Problem A.1 and Conjecture A.4]{Xu2021} (see Remark \ref{SovlePro}).

\begin{thm}\label{cor}
Let $L$ be a sectorial operator of type $\alpha$ ($0 \leq \alpha < \pi/2$) on $L^2(\real^d)$ satisfying \eqref{pt10}, \eqref{pt2} and \eqref{pt3}. Let $1 < q \leq2$. The followings are equivalent
    \begin{enumerate}[{\rm (i)}]
        \item $X$ is of martingale type $q$;
%        \item 
%        \[
%            \| f \|_{L^p(\real^d) } \lesssim_{p, d, \gamma, \beta} \| G_{q, L}(f) \|_p, \quad 1< p < \infty
%        \]
%        for all continuous functions $f: \real^d \to X$  with compact support.
     
        \item $X$ is of Lusin type $q$ relative to $\lk{e^{-tL}}_{t > 0}$. Moreover, we have the following estimate for the corresponding Lusin type constant, 
        \[
            \mathsf{L}^L_{t, q, p}(X) \lesssim_{\gamma, \beta} p \mathsf{M}_{t, p}(X), \quad 1 < p < \infty.
        \]
    \end{enumerate} 
    \end{thm}

Combining the main result in \cite{Ouyang2010}, a much stronger result than Theorem \ref{cor} involving the case $p=1,\infty$ will be presented in Corollary \ref{cor:strong}.

\

The paper is organized essentially as described above with a rigorous introduction of vector-valued tent space, Hardy spaces and BMO spaces in the next section.

\

{\bf Notation:}  In the following context, $X$ will be an arbitrary fixed Banach space without further elaboration. $X^*$ denotes the dual Banach space of $X$. Additionally, the positive real interval $\mathbb{R}_+ = (0, \infty)$ is equipped with the measure $\d t/t$ without providing additional explanations.

We will use the following convention: $ A \lesssim B $ (resp. $ A \lesssim_\alpha B $) means that $ A \leq CB $ (resp. $ A \leq C_\alpha B $) for some absolute positive constant C (resp. a positive constant $C_{\alpha}$ depending only on a parameter $\alpha$). $ A \approx B $ or $ A\approx_\alpha B $ means that these inequalities as well as their inverses hold. We also denote by $ \| \cdot \|_p $ the norm $ \| \cdot \|_{L^p(\real^d)} $ and by $\| \cdot \|_{L^p(X)}$ the norm $\| \cdot \|_{L^p(\real^d; X)}$ $(1 \leq p \leq \infty)$. 

\section{Preliminaries}\label{Preliminaries}

%%%%%Let $ ( \Omega, \F, \mu ) $ be a $ \sigma$-finite measure space of homogeneous type. We recall that $ L^0(\Omega, \F, \mu) $ denotes the vector space of all strongly $\mu$-measurable functions $f$ adapted to a $ \sigma $-algebra $\F$. Then we can define the martingales on $ ( \Omega, \F, \mu ) $.

%%%\begin{definition}
 % \begin{enumerate}[(i)]
  %  \item A family of sub-$\sigma$-algebras $ (\F_n)_{n\in I} $ of $\F$ is called a filtration in $ ( \Omega, \F, \mu ) $ if $ \F_m \subset \F_n $ whenever $ m,n \in I $ and $ m \leqslant n $, where $I$ is an ordered index set. The filtration is called $\sigma$-finite if $ \mu $ is $\sigma$-finite on each $\F_n$.
   % \item A family of functions $\sk{f_n}_{n \in I}$ in $ L^0( \Omega, \F, \mu ) $ is adapted to a filtration $ (\F_n)_{n\in I} $ if $ f_n \in L^0( \Omega, \F_n, \mu ) $ for all $n \in I$.
   % \item A family of functions $\sk{f_n}_{n \in I}$ in $ L^0( \Omega, \F, \mu ) $ is called a martingale with respect to a $\sigma$-finite filtration $(\F_n)_{n\in I}$ if it is adapted to $(\F_n)_{n\in I}$, and for all $m \leqslant n, f_n$ is $\sigma$-integral over $\F_m$ and satisfies
      %  $$
        %\ce (f_n|\F_m) = f_m.
       % $$
  %\end{enumerate}
%\end{definition}

%We recommend readers to \cite{Hy1} for more information.

%Assume $ (\F_n)_{n\geqslant 1} $ is a $\sigma$-finite filtration and $(f_n)_{n\geqslant 1}$ is a martingale respectively, denoted by $f$ for convenience.

\subsection{Functional calculus}\label{3.1}
We start with a brief introduction of some preliminaries around the holomorphic functional calculus (cf. \cite{Mcintosh1986}).
Let $0 \leq \alpha < \pi$. Define the closed sector in the complex plane $\com$ as
\[
    S_{\alpha} = \lk{ z \in \com: |\arg z| \leq \alpha },
\]
and $S_{\alpha}^0$ is denoted as the interior of $S_{\alpha}$.
Let $\gamma > \alpha$ and denote by $H(S_{\gamma}^0)$ the space of all holomorphic functions on $S_{\gamma}^0$. Define
\[
    H_{\infty}(S_{\gamma}^0) = \lk{ b \in H(S_{\gamma}^0) : \| b \|_\infty < \infty },
\]
where $\| b \|_\infty = \sup\lk{|b(z)|: z \in S_{\gamma}^0}$ and
\[
    \Psi(S_{\gamma}^0) = \lk{ \psi \in H(S_{\gamma}^0 ) : \exists\, s > 0 \text{ s.t. } |\psi(z)| \leq c|z|^s(1+|z|^{2s})^{-1} }.
\]
A densely defined closed operator $L$ acting on a Banach space $Y$ is called a \emph{sectorial operator of type $\alpha$} if for each $\gamma > \alpha$, $\sigma(L) \subset S_{\gamma}$ and 
\[
    \sup\lk{ \|z(z{\rm Id} - L)^{-1}\|_{B(Y)} : z \notin S_{\gamma} } < \infty,
\]
where $\| \cdot \|_{B(Y)}$ denotes the operator norm and ${\rm Id}$ the identity operator.

Assume that $L$ is a sectorial operator of type $\alpha$. Let $0 \leq \alpha < \theta < \gamma < \pi$ and $\Gamma$ be the boundary of $S_{\theta}$ oriented in the positive sense. For $\psi \in \Psi(S_{\gamma}^0)$, we define the operator $\psi(L)$ as
\[
    \psi(L) = \frac{1}{2\pi i} \int_{\Gamma} \psi(z)(z{\rm Id} - L)^{-1} \,\d z.
\]
By Cauchy's theorem, this integral converges absolutely in $B(Y)$ and it is clear that the definition is independent of the choice of $\theta$.
For every $t > 0$, denote by $\psi_t(z) = \psi(tz)$ for $z \in S_{\gamma}^0$, we have $\psi_t \in \Psi(S_{\gamma}^0) $. Set
\[
    h(z) = \int_0^{\infty} \psi(tz) \,\frac{\d t}{t}, \quad z \in S_{\gamma}^0.
\]
One gets that $h$ is a constant on $S_{\gamma}^0$, hence by the convergence lemma (cf. \cite[Lemma 2.1]{Cowling1996}),
\[
    h(L)x = \int_{0}^{\infty} \psi(tL)x \,\frac{\d t}{t} = cx,\quad x \in \D(L) \cap {\rm im}(L).
\]
By applying a limiting argument, the above identity extends to $\overline{ {\rm im}(L)}$. In particular, take $\psi(z) = z^2e^{-2z}$, then
\begin{equation}\label{funcal1}
    \int_{0}^{\infty} -tLe^{-tL}(-tLe^{-tL})x \,\frac{\d t}{t} = \frac{1}{4}x,\quad x \in \overline{ {\rm im}(L) },
\end{equation}
which will be useful later.  We refer the reader to \cite{Haase2005} for more information on functional calculus.

\subsection{Main assumptions}\label{mainassumption}
Throughout the paper, we assume $L$ is a sectorial operator of type $\alpha$ ($0 \leq \alpha < \pi/2$) on $L^2(\real^d)$ such that the kernels $\{K(t, x,y)\}_{t>0}$ of $\{e^{-tL}\}_{t>0}$ satisfy assumptions \eqref{pt10}, \eqref{pt2} and \eqref{pt3} with $ \beta > 0, 0 < \gamma \leq 1  $. It is well-known that such an $L$ generates a holomorphic semigroup  $e^{-zL}  $ with $0\leq |\Arg(z)| <\pi/2-\alpha$  (cf. \cite[Chapter 3, 3.2]{Haase2005}). Let $\{k(t, x,y)\}_{t>0}$ be the kernels of $\{-tLe^{-tL}\}_{t>0}$ and it is easy to see
\[
    k(t, x, y) = t \partial_tK(t, x, y).
\]

The following lemma is justified in \cite[Lemma 6.9]{DXT}
\begin{lemma}\label{2.8}
Let $ L $ be an operator satisfying \eqref{pt10} and \eqref{pt2} with $\beta>0$, $ 0 <\gamma \leq 1  $. Then

{\rm (i)} there exist constants $ 0 < \beta_1 < \beta $, $ 0< \gamma_1 < \gamma $ and $c>0$  such that

    \begin{equation}\label{qt1}
         |k(t, x, y)| \leq \frac{ct^{\beta_1}}{(t + |x-y|)^{d+\beta_1}},
    \end{equation}
and
\[
    | k(t, x+h, y) - k(t, x, y) | + | k(t, x, y+h) - k(t, x, y) | \leq \frac{c|h|^{\gamma_1} t^{\beta_1}}{(t + |x-y|)^{d+\beta_1+\gamma_1}}
\]
whenever $ 2|h| \leq t + |x-y| $;

{\rm (ii)} for $\alpha < \theta < \pi/2$, there exist positive constants $ 0 < \beta_2< \beta $, $ 0< \gamma_2 < \gamma $ and $c>0$ such that for any $|\arg z| < \pi/2 - \theta$,
\begin{equation}\label{pt1}
    |K(z, x, y)| \leq \frac{c |z|^{\beta_2}}{(|z| + |x - y|)^{d + \beta_2}}
\end{equation}
and
\[
    | K(z, x+h, y) - K(z, x, y) | + | K(z, x, y+h) - K(z, x, y) | \leq \frac{c|h|^{\gamma_2} |z|^{\beta_2}}{(|z| + |x-y|)^{d+\beta_2+\gamma_2}}
\]
whenever $ 2|h| \leq |z| + |x-y| $.
\end{lemma}

\begin{remark}\label{r2.1}
By \cite[Lemma 2.5]{Coulhon2000}, the estimate \eqref{pt1} implies that for all $k \in \nat$, $t > 0$ and almost everywhere $x, y \in \real^d$,
    \begin{equation}\label{derpt}
        \jdz{ t^k \partial^k_t K(t, x, y) } \leq \frac{c t^{\beta_2} }{(t + |x - y|)^{d + \beta_2}  }.
    \end{equation}
        \end{remark}
        
        {\noindent {\bf Convention.}} {To simplify notation, we will write below $\gamma$, $\beta$ instead of $\gamma_1, \ \beta_1$ and $\gamma_2, \ \beta_2$ appearing in Lemma \ref{2.8}, and it should not cause any confusion.}
        
        \iffalse
        {\color{blue}It is easy to check that $e^{-tL}$ is a regular operator, whose regular norm is independent of $t$. Recall that a regular operator $T$ on $L^p(\real^d)$ is \emph{regular} if there exists a constant $c$ such that
        \[
            \sup_k \| T(f_k) \|_p \leq c\big\| \sup_k|f_k| \big\|_p.
        \]
        for any finite sequence $\lk{f_k}_{k \geq 1} \subset L^p(\real^d)$.
        
        Let $1 \leq p \leq \infty$. Now we verify that $e^{-tL}$ is a regular operator on $L^p(\real^d)$. Note that
        \[
        \begin{aligned}
            \sup_k \| e^{-tL}(f_k) \|_p^p & = \sup_k\int_{\real^d} \jdz{ \int_{\real^d} K(t, x, y)f_k(y)\,\d y }^p \,\d x \\
            & \leq \int_{\real^d} \mk{ \int_{\real^d} |K(t, x, y)|\sk{\sup_k |f_k(y)| }\,\d y }^p \,\d x \\
            & \leq c^p\int_{\real^d} \mk{ \int_{\real^d} \frac{t^{-d}}{( 1 + t^{-1}|x - y| )^{d+\beta}}\sk{\sup_k |f_k(y)| }\,\d y }^p \,\d x\\
            & \lesssim_\beta c^p \int_{\real^d} \mk{ \int_{\real^d} \frac{t^{-d}}{( 1 + t^{-1}|x - y| )^{d+\beta}}\sk{\sup_k |f_k(y)| }^p\,\d y }\,\d x \\
            & \lesssim_\beta c^p \int_{\real^d} \sk{\sup_k |f_k(y)| }^p\,\d y = c^p  \big\| \sup_k|f_k| \big\|_p^p.
        \end{aligned} 
        \]
        This argument also holds for $p = \infty$.
        }
        \fi

One can verify that $\lk{e^{-tL}}_{t > 0}$ is a family of regular operators on $L^p(\real^d)$ for $1 \leq p \leq \infty$. Then it is well-known (see e.g. \cite{Xu2021}) that the semigroup $\{e^{-tL}\}_{t>0}$ extends to $L^p(\real^d; X)$ ($1\leq p \leq \infty$), which is the space of all strongly measurable functions $f: \real^d \to X$ such that $\|f(x)\|_X \in L^p(\real^d)$. The resulting semigroup is still denoted by $\{e^{-tL}\}_{t>0}$ without causing confusion. To well define the vector-valued BMO type spaces, we need more notations. For $\e > 0$, define
\[
    \N_\e = \lk{ f\in L^1_{{\rm loc}}(\real^d; X) : \exists\, c > 0 \text{ such that } \int_{\real^d} \frac{ \|f(x)\|_X }{(1 + |x|)^{d + \e}} \,\d x\leq c },
\]
equipped with norm defined as the infimum of all the possible constant $c$. Then $\N_\e$ is a Banach space (cf. \cite{DXT}). For a given generator $L$, let $\Theta(L) = \sup\lk{ \beta_2 > 0 : \eqref{pt1} \text{ holds} } $. Then we define
\[
    \N=\begin{cases}
        \N_{\Theta(L)}, & \ \mathrm{if}\ \Theta(L) < \infty ;\\
        \bigcup_{0 < \e < \infty} \N_\e, & \ \mathrm{if}\ \Theta(L) = \infty.
    \end{cases}  
\] 
It is clear that $L^p(\real^d;X) \subset \N$ for all $1 \leq p \leq \infty$. Moreover, 
By the definition of $\N$ and Remark \ref{r2.1}, we know that the operators $e^{-tL}$ and $tLe^{-tL}$ are well-defined on $\N$.

Denote by $\mathbf{F}_L$ the fixed point space of $\{e^{-tL}\}_{t>0}$ on $\N$, namely
\[
    \mathbf{F}_L = \lk{ f \in \N: e^{-tL}(f) = f,\ \forall\, t > 0}.
\]
It is well-known that $\mathbf{F}_L$ coincides with the null space of $\{tLe^{-tL}\}_{t>0}$, and the resulting quotient space is defined as $\N_L := \N / \mathbf{F}_L$. For $1\leq p<\infty$, the fixed point subspace of $L^p(\real^d; X)$ is $\mathbf{F}_L\cap L^p(\real ^d; X) = \lk{0}$ (see \cite[Theorem 6.10]{DXT}); {in other words, the projection from $L^p(\real ^d; X)$ to the fixed point subspace for all $1\leq p<\infty$ is $0$. See e.g. \cite{MTX, Xu2021} for more information on this projection.

\begin{remark}\label{rk:Lstar}
    Let $L^*$ be the adjoint operator of $L$. Then $L^*$ is also a sectorial operator with the kernels of $\{e^{-tL^*}\}_{t>0}$ satisfying \eqref{pt10}, \eqref{pt2} and \eqref{pt3} (cf. \cite[Theorem 6.10]{DXT}).
\end{remark}

    %Different from \cite{DXT}, we do not assume that $L$ satisfies the $H_{\infty}$-calculus on $L_X^p(\real^d)$ since in many important cases, this condition holds only when the underlying Banach space $X$ is UMD, such as $L = \sqrt{\Delta}$ or $\Delta$ (cf. \cite{Xu2021}).  

\subsection{Vector-valued tent, Hardy and BMO spaces}
In this subsection, we introduce several spaces including vector-valued tent spaces, vector-valued Hardy and BMO spaces associated with a generator $L$. %The sectorial operator $L$ is assumed to satisfy \eqref{pt1} and the parameter $q$ is always assumed to be in $(1, \infty)$.

\subsubsection{Vector-valued Tent spaces}
We first introduce vector-valued tent spaces. We denote by $ \real^{d+1}_+ $ the usual upper half-space in $ \real^{d+1} $ i.e. $ \real^d \times (0,\infty)$. Let $ \Gamma(x) = \{ (y,t) \in \real^{d+1}_+ : |y-x|<t \} $ denote the standard cone with  vertex at $x$. For any closed subset $ F \subset \real^d $, define $ \R(F) = \bigcup_{x\in F} \Gamma(x) $. If $ O \subset \real^d $ is an open subset, then the tent over $ O $, denoted by $ \wh{O} $, is given as $ \wh{O} = \sk{\R(O^C)}^C $. 

For any strongly measurable function $f: \real^{d+1}_+ \to X$, we define two operators as follows:
\[
    \A_q(f)(x) = \sk{\int_{\Gamma(x)} \| f(y, t) \|_X^q \frac{\d y \d t}{t^{d+1}} }^{\frac{1}{q}},\quad \mathcal{C}_{q}(f)(x) = \sup_{x \in B} \sk{\frac{1}{|B|} \int_{\wh{B}} \|f(y,t)\|_X^q \,\frac{\d y\d t}{t}  }^\frac{1}{q},
\]
where the supremum runs over all balls $B$ in $\real^d$.
\begin{definition}
Let $1 \leq p < \infty$ and $1<q<\infty$. The vector-valued tent space {$T_q^p( \real^{d+1}_+; X)$} is defined as the subspace consisting of all strongly measurable functions $f: \real^{d+1}_+ \to X$ such that \[
    \| f \|_{T_q^p(X)} := \| \A_q(f) \|_p < \infty,
\] 
and $T_q^{\infty}( \real^{d+1}_+; X)$ is defined as the subspace of all strongly measurable functions $g: \real^{d+1}_+ \to X$ such that
\[
    \| g \|_{T_q^{\infty}(X)} := \| \mathcal{C}_q(g) \|_\infty < \infty.
\]
\end{definition}

Let $C_c(\real^{d+1}_+)\otimes X$ be the space of finite linear combinations of elements from $C_c(\real^{d+1}_+)$ and $X$. The following density follows from the standard arguments (see e.g. \cite{Harboure1991}), and we omit the details. %and also the case $p=\infty$ by the duality (see Lemma \ref{dualtent} below).

\begin{lemma}\label{loc}
Let $X$ be a Banach space and $1<q<\infty$. Then
    $C_c(\real^{d+1}_+)\otimes X$ is norm dense in $T^p_q(\real^{d+1}_+; X)$ for $1 \leq p <\infty$, {and weak-$\ast$ dense in $\sk{ T^1_{q'}(\real^{d+1}_+; X^*) }^*$}.
\end{lemma}
% {\color{blue} for $1 < p <\infty$ and the classical argument for $p=1$. To be discussed}, one gets the following lemma.

%{\color{blue}This proof may be standard. Let $\eta_\e: \real^{d+1}_+ \to \com$ be a standard mollifier on $\real^{d+1}_+$. For any $g \in T^\infty_q(\real^{d+1}; X)$, we can approximate $g$ by $g \star \eta_\e$. For any $f \in C_c(\real^{d+1}_+) \ot X^*$,
%\[
%    \int_{\real^{d+1}_+} \lara{g\star \eta_\e, f}_{X \times X^*}\, \frac{\d x\d t}{t} = \int_{\real^{d+1}_+} \lara{g, f\star \eta_\e}_{X \times X^*}\, \frac{\d x\d t}{t} \to \int_{\real^{d+1}_+}\lara{g, f}_{X \times X^*}\, \frac{\d x\d t}{t}.
%\]
%Hence we obtain $g\star \eta_\e \to g$ in the weak-$\ast$ sense since $C_c(\real^{d+1}_+) \ot X^*$ is dense in $T^1_q(\real^{d+1}; X^*)$. 
%}{\color{red}YOU have used the duality whose proof appeared later. Need to discuss}

\subsubsection{Vector-valued Hardy spaces}
Given a function $ f \in \N_L $, the $q$-variant of Lusin area integral function of $f$ associated with $L$ is defined by
\[
  S_{q,L}(f)(x) = \sk{ \int_{\Gamma(x)} \|tLe^{-tL}(f)(y)\|_X^q\,\frac{\d y \d t}{t^{d+1}} }^{\frac{1}{q}};
\]
and the $q$-variant of Littlewood-Paley $g$-function is defined by
\[
    G_{q,L}(f)(x) =  \sk{ \int_{0}^{\infty} \|tLe^{-tL}(f)(x)\|_X^q\,\frac{\d t}{t} }^{\frac{1}{q}}.
\]

\begin{definition}
Let $ 1 \leq p < \infty$ and $1<q<\infty$. We define the vector-valued Hardy space $H_{q,L}^p(\real^d; X)$ associated with $L$ as 
\[
    H_{q,L}^p(\real^d; X) = \lk{ f \in \N_L: S_{q,L}(f) \in L^p(\real^d)  },
\]
equipped with the norm
\[
    \| f \|_{H^p_{q, L}(X)} = \| S_{q,L}(f) \|_p.
\]
\end{definition}
It is easy to check that $H^p_{q, L}(\real^d; X)$ is a Banach space from the definition of $\N_L$. The space $H^p_{q, L}(\real^d; X)$ has deep connection with the vector-valued tent space, namely, a strongly measurable function $f \in \N_L$ belongs to $H^p_{q, L}(\real^d; X)$ if and only if {$\mathcal Q(f) \in T^p_q(\real^{d+1}_+;X)$ where $\Q(f)(x,t)=-tLe^{-tL}(f)(x)$}. Moreover,
\[
    \| f \|_{H^p_{q, L}(X)} = \| \Q(f) \|_{T^p_q(X)}.
\]

%We introduce the vector-valued Hardy space $H_{q,L}^p(X)$ ($1 \leq p < q$) associated to an operator $L$ as 
%\[
%    H^p_{q, L}(X) = \lk{ f \in L^1(X): S_{q, L}(f) \in L^p(\real^d)  }
%\]
%where the norm is defined as
%\[
%    \| f \|_{H^p_{q, L}} = \| S_{q, L}(f) \|_p.
%\]
%It is evident that $H^p_{q, L}(X)$ forms a Banach space. 

\subsubsection{Vector-valued BMO spaces}
\begin{definition}
Let $1 \leq p \leq \infty$ and $1<q<\infty$. We define the vector-valued BMO space $BMO_{q,L}^p(\real^d; X)$ associated with $L$ as 
\[
    BMO_{q,L}^p(\real^d; X) = \lk{ f \in \N_L: \| \mathcal{C}_q(\Q(f)) \|_p < \infty }
\]
equipped with the norm
\[
    \| f \|_{BMO_{q,L}^p(X)} = \| \mathcal{C}_q(\Q(f)) \|_p.
\]
In particular, for $p = \infty$, we denote it by $BMO_{q, L}(\real^d; X)$ for short.
\end{definition}

It is easy to verify that $BMO^p_{q, L}(\real^d; X)$ equipped the the norm $\| \cdot \|_{BMO^p_{q, L}(X)}$ is a Banach space {from the definition of $\N_L$}.

The vector-valued Hardy and BMO spaces enjoy the similar relationship as the scalar-valued ones (see e.g. \cite{Coifman1985}). We collect them below with a brief explanation.

\begin{lemma}
Let $X$ be any fixed Banach space and $1<q<\infty$. One has for $f\in C_c(\real^{d+1}_+)\otimes X$, 
\begin{equation}\label{A-C}
    \| \mathcal{C}_{q}(f) \|_{p} \lesssim \sk{\frac{p}{p-q}}^{\frac{1}{q}} \| \A_{q}(f) \|_{p}, \quad q < p \leq \infty,
\end{equation}
and
\begin{equation}\label{C-A}
    \| \A_q(f) \|_p \lesssim q^{\frac{p}{q}} \| \mathcal{C}_q(f) \|_p, \quad 1 \leq p < \infty.
\end{equation} 
Therefore, we have for $1\leq p\leq q$, 
\[
    BMO_{q, L}^p(\real^d; X) \subset H^p_{q, L}(\real^d; X)
\]
and for $q < p < \infty$,
\[
    H^p_{q, L}(\real^d; X) = BMO_{q, L}^p(\real^d; X)
\]
with equivalent norms.
\end{lemma}

\begin{proof}
    Given an $X$-valued function $f$ defined on $\real^{d+1}_+$, we consider the scalar-valued function $ \wt{f}(x, t) = \| f(x, t) \|_X$. Then one may apply \eqref{A-C} and \eqref{C-A} in the case $X=\mathbb C$ for $\wt{f}$ (see e.g.\cite[Theorem 3]{Coifman1985}) to obtain \eqref{A-C} and \eqref{C-A} for general $X$. Thus by using the operator $\mathcal Q$ and the density in Lemma \ref{loc}, for any $f \in BMO_{q, L}^p(\real^d; X)$ ($1 \leq p \leq q$), we get
    \[
        \| f \|_{H^p_{q, L}(X)} = \| \A_q(\Q(f)) \|_p \lesssim q^{\frac{p}{q}}\| \mathcal{C}_q(\Q(f)) \|_p = \| f \|_{BMO_{q, L}^p(X)},
    \]
 and the same argument works for $q < p < \infty$.
\end{proof}

\begin{remark}
    In particular, $BMO_{q, L}(\real^d; X)$ is closely related to the Carleson measure. Recall that a scalar-valued measure $ \mu $ defined on $ \real^{d+1}_+ $ is a Carleson measure if there exists a constant $ c $ such that for all balls $ B $ in $ \real^d $,
    \[
        |\mu(\wh{B})| \leq c|B|,
    \]
    where $ \wh{B} $ is the tent over $ B $. The norm is defined as
    \[
        \| \mu \|_c = \sup_{B} |B|^{-1} |\mu( \wh{B} )|,
    \]
    where the supremum runs over all the balls in $\real^d$.

    For a vector-valued function $f \in \N_L$, we define a measure $\mu_{q, f}$ as
    \[
        \mu_{q, f}(x, t) = \frac{\| \Q(f)(x,t) \|_X^q \d x\d t }{t}.
    \]
    Then $f$ belongs to $ BMO_{q,L}(\real^d; X)$ if and only if $\mu_{q, f}$ is a Carleson measure, and moreover
    \[
        \| f \|_{BMO_{q, L}(X)} = \| \mu_{q, f} \|^{\frac{1}{q}}_c.
    \]
\end{remark}

\section{Theory of vector-valued tent spaces and two key linear operators}\label{Tentspace}
In this section, we will first present the basic theory of vector-valued tent spaces such as atomic decomposition, interpolation and duality, and then introduce two important linear operators $\mathcal K$ and $\pi_L$ which enable us to exploit the basic theory of tent spaces to investigate in later sections vector-valued Wilson's square functions and Theorem \ref{main2}. %The sectorial operator $L$ is assumed to satisfy \eqref{pt1} without further explanation.

Note that if the underlying Banach space $X$ has some geometric property such as UMD, then the vector-valued tent space theory have been established in the literature \cite{Hytonen2008, Kemppainen2014,Kemppainen2016}. In the present paper, we observe that the theory of vector-valued tent space holds for any Banach space; and this is quite essential for the applications in the present paper.  
% We would like to mention that Hyt\"{o}nen et al. in \cite{Hytonen2008} and Kemppainen in \cite{Kemppainen2014,Kemppainen2016} have developed nice theorems on vector-valued tent spaces on certain measure space associated with stochastic integrals. But most of their arguments and results rely on the UMD property for the underlying Banach space.  

\subsection{Basic theory of vector-valued tent spaces}
We begin this subsection by presenting  the atomic decomposition of tent space in the context of vector-valued context. It has been established in \cite[Theorem 4.5]{Kemppainen2014}, for the completeness of this article, we will attach the proof. Recall that a strongly measurable function $a : \real^{d+1}_+ \to X$ is called an $(X, q)$-atom if
\begin{enumerate}
    \item $\supp\, a \subset \wh{B}$ where $B$ is a ball in $\real^d$;
    \item $ \sk{ \int_{\real^{d+1}_+} \|a(x, t)\|_X^q \, \frac{\d x \d t}{t} }^{\frac{1}{q}} \leq |B|^{\frac1q - 1}. $
\end{enumerate}

\begin{lemma}\label{atmicDectent}
Let $X$ be any fixed Banach space and $1<q<\infty$.    For each $f \in T_q^1(\real^{d+1}_+; X)$, there exists a sequence of complex numbers $\{\lambda_k \}_{k \geq 1}$ and $(X, q)$-atoms $a_k$ such that
    \[
        f = \sum_{k \geq 1}\lambda_k a_k, \quad \| f \|_{T_q^1(X)} \approx \sum_{k \geq 1}|\lambda_k|.
    \]
    %The series converges almost everywhere $(x, t) \in \real^{d+1}_+$.
\end{lemma}

\begin{proof}
Let $a$ be an $(X, q)$-atom and $\supp\, a \subset \wh{B}$ where $B = B(c_B, r_B)$ with center $c_B$ and radius $r_B$. If $\Gamma(x) \cap \wh{B} \neq \emptyset$, there exists $(y, t) \in \Gamma(x) \cap \wh{B}$. Then we have $|x - c_B| \leq |x - y| + |y - c_B| < t + r_B < 2r_B$. By H\"{o}lder's inequality and Fubini's theorem,
\[
    \| a \|_{T^1_q(X)} =\int_{2B}\sk{\int_{\Gamma(x)} \| a(y, t) \|_X^q \,\frac{\d y\d t}{t^{d+1}}}^{\frac{1}{q}}\d x \lesssim |2B|^{1-\frac{1}{q}}\sk{\int_{\real^{d+1}_+}\| a(y, t) \|_X^q \,\frac{\d y\d t}{t}}^{\frac{1}{q}} \lesssim 1.
\]
Therefore any $(X, q)$-atom belongs to $T^1_q(\real^{d+1}_+;X)$. 

Let $0 < \lambda < 1/2$. We define two sequences of open sets $\lk{O_k}_{k \in \ent}$ and $\lk{O_k^*}_{k \in \ent}$ as
\[
    O_k= \lk{ x \in \real^d: \A_q(f)(x) > 2^k },\quad O_k^*=\lk{ x\in \real^d: M(\1_{O_k})(x) > 1-\lambda },
\]
where $M(\1_{O_k})$ is the centered Hardy-Littlewood maximal function. It is clear that both $O_k$ and $O^*_k$ have finite measure. Additionally, the following properties hold: $O_{k+1} \subset O_k$, $O^*_{k+1} \subset O^*_k$ and $|O^*_k| \leq C_\lambda|O_k|$ (see e.g. \cite{Coifman1985}).
    
We follow a similar construction as in \cite{Kemppainen2014}. The {Vitali covering lemma} and \cite[Lemma 4.2]{Kemppainen2016} assert that for each $O_k^*$, there exist disjoint balls $B_k^j \subset O_k^*$ ($j \geq 1$) such that 
\[
    \wh{O}^*_k \subset \bigcup_{j \geq 1} \wh{5B}_k^j, \quad \sum_{j \geq 1}|B_k^j| \leq |O_k^*|.
\]
With this setup, we proceed to define a family of functions $\chi_k^j$ by the partition of unity:
\[
    0 \leq \chi_k^j \leq 1,\ \sum_{j \geq 1}\chi_k^j = 1 \text{ on } \wh{O}_k^* \ \text{ and } \ \supp\, \chi_k^j \subset \wh{5B}_k^j.
\]
Therefore 
\[
    f =\sum_{k \in \ent} f_k = \sum_{k\in \ent}\sum_{j \geq 1} \chi_{k}^jf_k = \sum_{k\in \ent}\sum_{j \geq 1} \lambda_k^j a_k^j,
\]
where
\begin{equation}\label{Def_f_k}
    f_k = f \1_{\wh{O}^*_k\setminus \wh{O}^*_{k+1} },\ \lambda_k^j = |5B_k^j|^{\frac{1}{q'}}\sk{ \int_{ 5B_k^j } \A_q(f_k)^q(x) \,\d x }^{\frac{1}{q}}, \ a_k^j = \frac{\chi_k^jf_k}{\lambda_k^j}.
\end{equation}
Now we only need to show that each $a_k^j$ is an $(X, q)$-atom and 
\[
    \sum_{k \in \ent} \sum_{j \geq 1}|\lambda_k^j| \lesssim \| f \|_{T^1_q(X)}.
\]
It is clear that $\supp\, a_k^j \subset \wh{5B}_{k}^j $. Furthermore,
\[
\begin{aligned}
    \| a \|_{L^q(\real^{d+1}_+; X)}^q & \leq |5B_k^j|^{1-q}\| \A_q(f_k)\1_{5B_k^j} \|_q^{-q} \sk{\int_{\wh{5B}_k^j} \| f_k(y, t) \|_X^q \,\frac{\d y\d t}{t} } \\ 
    & \leq |5B_k^j|^{1-q}\| \A_q(f_k)\1_{5B_k^j} \|_q^{-q} \sk{\int_{5B_k^j} \sk{\A_q(f_k)(x)}^q \,\d x }\\ & = |5B_k^j|^{1-q}.
\end{aligned}  
\]
Hence each $a_k^j$ is an $(X, q)$-atom. 

According to \cite[Lemma 5]{Coifman1985}, it is known that $\A_q(f_k)$ is supported in $O^*_k \setminus O_{k+1}$, then we deduce that $\A_q(f_k)(x) \leq 2^{k+1}$ by definition. Thus
\[
    \sum_{k \in\ent}\sum_{j \geq 1} |\lambda_k^j| \leq \sum_{k \in\ent}\sum_{j \geq 1} |5B_k^j|^{\frac{1}{q'}} 2^{k+1} |5B_k^j|^{\frac{1}{q}} \leq \sum_{k \in \ent} 2^{k + 1}|O_k^*| \leq \sum_{k \in \ent} 2^{k + 1}C_\lambda|O_k|.
\]
However, $\A_q(f)(x) > 2^{(k+m)}$ on $O_{k+m}$, then
\[
    2^{k}|O_k| = \int_{O_k}2^{k} \,\d x = \sum_{m = 0}^{\infty} \int_{O_{k+m}\setminus O_{k+m+1}} 2^{k}\,\d x \leq \sum_{m = 0}^{\infty} 2^{-m}\int_{O_{k+m}\setminus O_{k+m+1}} \A_q(f)(x)\,\d x.
\]
Hence
\[
    \sum_{k \in \ent} 2^{k + 1}C_\lambda|O_k| \leq \sum_{m = 0}^{\infty}\sum_{k \in \ent} 2^{-m+1}C_\lambda\int_{O_{k+m}\setminus O_{k+m+1}} \A_q(f)(x)\,\d x \lesssim \| f \|_{T^1_q(X)}.
\]
We complete the proof.
\end{proof}

\begin{remark}
From the atomic decomposition of $T^1_q(\mathbb R^{d+1}_1; X)$---Lemma \ref{atmicDectent}, one may conclude a molecule decomposition of the corresponding Hardy space. This might have further applications, and we include it in the Appendix.
\end{remark}

The following lemma is the complex interpolation theory of vector-valued tent spaces.
\begin{lemma}\label{Intpo}
   Let $X$ be any fixed Banach space, $1<q<\infty$ and $1 \leq p_1 < p < p_2 < \infty$ such that $1/p = (1-\theta)/p_1 + \theta/p_2$ with $0\leq\theta\leq1$. Then
   \[
        [T^{p_1}_q(\real^{d+1}_+;X), T^{p_2}_q(\real^{d+1}_+;X)]_\theta = T^p_q(\real^{d+1}_+;X),
   \]
   with equivalent norms, where $[\cdot, \cdot]_\theta$ is the complex interpolation space. More precisely, for $f\in C_c(\real^{d+1}_+)\otimes X$, one has {
   \[
        \|f \|_{T^p_q(X)} \lesssim \| f \|_{[T^{p_1}_q(\real^{d+1}_+;X), T^{p_2}_q(\real^{d+1}_+;X)]_\theta} \lesssim p^{\frac{1}{q}} \| f \|_{T^p_q(X)}.
   \]
   }
\end{lemma}
\begin{proof}
   %Because of the unavailability of vector-valued Calder\'{o}n-Zygmund theory in our setting for general $X$, the arguments in \cite{Hytonen2008, Kemppainen2014,Kemppainen2016, Harboure1991} fail and We shall adapt the classical argument in \cite[Lemma 4, Lemma 5]{Coifman1985}. The construction and argument are indeed the same with a slight change that we use the norm in place of the absolute value. The constant is given by a weighted inequality, we refer readers to \cite[Lemma 5]{Coifman1985} for details. 
For the interpolation theory, we introduce two important operators, which allow us to relate $T^p_q(\real^{d+1}_+; X)$ with $L^p(\real^d; E)$ with $E$ being the Banach space $L^q(\real^{d+1}_+; X)$ equipped with the measure $\d x\d t/t^{d+1}$.  The first operator is defined as
\[
    i(f)(x, y, t) = \1_{\Gamma(x)} (y, t) f(y, t),
\]
for $f \in T^p_q(\real^{d+1}_+; X)$. Then it is clear that $\| i(f) \|_{ L^p(E) } = \| f \|_{T^p_q(X)}$. Denote by  $\wt{T}_q^p$ the range of the operator $i$. Now we introduce another operator $N$ given by
\[
    N(f)(x, y, t) = \1_{\Gamma(x)}(y, t) \frac{1}{w_dt^d} \int_{ |z - y|< t } f(z, y, t) \,\d z,
\]
where $w_d$ is the volume of the $d$-dimensional unit ball. It is known that $N$ is a continuous projection from $L^p(\real^d; E)$ onto itself with range $\wt{T}_q^p$  for $1 < p < \infty$ (cf. \cite{Harboure1991}).  Consider the maximal operator 
\[
    M_1(f)(x, y, t) = \sup_{x \in B} \frac{1}{|B|} \int_B \|f(z, y, t)\|_X\, \d z,
\]
where the supremum is taken over all balls $B$ in $\real^d$. It is known from the maximal inequalities (see e.g. \cite[Chapter \Rmnum{2}]{Stein1993}) that $M_1$ is bounded on $L^p\sk{ \real^d; L^q(\real^{d+1}; \d y\d t/t^{d+1}) }$ for $1 < p < \infty$; in particular, we view $\|f\|_X$ as a scalar-valued function in $L^p\sk{ \real^d; L^q(\real^{d+1}; \d y\d t/t^{d+1}) }$, then
\[
    \| M_1(f) \|_{L^p\sk{ \real^d; L^q\sk{ \real^{d+1}_+, \frac{\d y\d t}{t^{d+1}}} } } = \| M_1(\|f\|_X) \|_{L^p\sk{ \real^d; L^q\sk{ \real^{d+1}_+, \frac{\d y\d t}{t^{d+1}}} } } \lesssim p^{\frac{1}{q}}\| f \|_{L^p(E)}, \quad q \leq p < \infty.
\]
Then we deduce from the definition of $N$ that
\[
    \| N(f)(x, y, t) \|_X \leq \1_{\Gamma(x)}(y, t) \frac{1}{|B(y, t)|} \int_{ B(y, t) } \|f(z, y, t)\|_X \,\d z \leq M_1(f)(x, y, t).
\]
Therefore
\[
    \| N(f) \|_{L^p(E)} \leq \| M_1(f) \|_{L^p\sk{ \real^d; L^q\sk{ \real^{d+1}_+; \frac{\d y\d t}{t^{d+1}}} }} \lesssim p^{\frac{1}{q}} \| f \|_{L^p(E)}, \quad q \leq p < \infty.
\]
We denote by $F$ the Banach space $L^{q'}(\real^{d+1}_+; X^*)$ equipped with the measure $\d x\d t/t^{d+1}$. Then it is clear that $F \subset E^*$ and $F$ is norming for $E$. For $1 < p < q$, we have 
\[
\begin{aligned}
    \| N(f) \|_{L^p(E)} & = \sup_{g} \jdz{ \int_{\real^d}\int_{\real^{d+1}_+} \lara{ N(f)(x, y, t), g(x, y, t) }_{X\times X^*}\,\frac{\d y\d t}{t^{d+1}}\d x }\\
    & = \sup_{g} \jdz{ \int_{\real^d}\int_{\real^{d+1}_+} \lara{ f(x, y, t), N(g)(x, y, t) }_{X\times X^*}\,\frac{\d y\d t}{t^{d+1}}\d x }\\
    & \leq \| f \|_{L^p(E)} \| N(g) \|_{L^{p'}(F)} \lesssim p'^{\frac{1}{q'}}\| f \|_{L^p(E)}\| g \|_{L^{p'}(F)},
\end{aligned}
\]
where the supremum is taken over all $g$ in the unit ball of $L^{p'}(\real^d; F)$. We conclude
\begin{equation}\label{esN}
    \| N(f) \|_{L^p(E)} \lesssim \max\lk{ p^{\frac{1}{q}}, p'^{\frac{1}{q'}} }\| f \|_{L^p(E)}, \quad 1 < p < \infty.
\end{equation}

Now we turn to the interpolation theory. The proof of the case $1 < p_1 < p_2 < \infty$ follows from \cite{Harboure1991} by virtue of the immersion $i$ and the projection $N$. 

For the case $p_1 = 1$, we adapt the classical argument as in \cite[Lemma 4, Lemma 5]{Coifman1985}. Since the immersion $i$ is an isometry, the exactness of the exponent $\theta$ of complex interpolation functor reads that
\[
\begin{aligned}
    & \| i(f) \|_{[L^1(\real^d; E), L^{p_2}(\real^d; E)]_\theta} \\
    \leq\ & \| i \|_{T^1_q(\real^{d+1}_+; X) \to L^1(\real^d; E)}^{1- \theta}\| i \|_{T^{p_2}_q(\real^{d+1}_+; X) \to L^{p_2}(\real^d; E)}^{\theta} \| f \|_{[T^1_q(\real^{d+1}_+; X), T^{p_2}_q(\real^{d+1}_+; X)]_\theta} \\
    \leq\ & \| f \|_{[T^1_q(\real^{d+1}_+; X), T^{p_2}_q(\real^{d+1}_+; X)]_\theta}.
\end{aligned}
\]
By the interpolation theory of vector-valued $L^p$ spaces, (see e.g. \cite{Blasco1989}), we have
\[
    \| i(f) \|_{[L^1(\real^d; E), L^{p_2}(\real^d; E)]_\theta} = \| i(f) \|_{L^p(E)} = \| f \|_{T^p_q(X)}.
\]
Thus 
\[
    \| f \|_{T^p_q(X)} \leq \| f \|_{[T^1_q(\real^{d+1}_+; X), T^{p_2}_q(\real^{d+1}_+; X)]_\theta}.
\]

For the reverse direction, let $f \in T^p_q(\real^{d+1}_+; X)$ and $\| f \|_{T^p_q(X)} = 1$. By taking into account the atomic decomposition of $T^1_q(\real^{d+1}_+; X)$---Lemma \ref{atmicDectent}, we define the interpolation functor $F$ as
\[
    F(z) = \sum_{k \in \ent} 2^{k(\alpha(z)p - 1)}f_k,
\]
where $\alpha(z) = 1-z + z/p_2$ and $f_k$ is defined in \eqref{Def_f_k}. We have $F(\theta) = f$. Then the proof can be then conducted in the same way as in \cite[Lemma 5]{Coifman1985}, we omit the details.
\end{proof}

We now provide a characterization of $T^p_q(\real^{d+1}_+; X)$-norm. It belongs to the norming subspace theory of vector-valued $L^p$-spaces, see e.g. \cite[Chapter \Rmnum{2}, Section 4]{Dinculeanu2014}. The proof is in spirit the same as the scalar-valued case (cf. \cite[Theorem 2.4]{Harboure1991} and \cite[Theorem 1]{Coifman1985}, but we include a proof here to provide explicit orders for later applications.

Let $(\Omega, \F, \mu)$ be a measure space. Recall that a Banach space $X$ has the \emph{Radon-Nikod{\'y}m property} with respect to $(\Omega, \F, \mu)$ if for each $\mu$-continuous vector-valued measure $\nu : \F \to X$ of bounded variation, there exists $g \in L^1(\Omega; X)$ with respect to the measure $\mu$ such that 
\[
    \nu(E) = \int_{E} g \,\d \mu, \quad \forall\, E \in \F.
\]
In the following context, we call a Banach space has the Radon-Nikod{\'y}m property for short when there is no ambiguity. We refer readers to \cite[Chapter \Rmnum{3}]{Diestel1977} for more details.

\begin{lemma}\label{dualtent}
Let $X$ be any fixed Banach space and $1<q<\infty$. The space
    $T^{p'}_{q'}(\real^{d+1}_+; X^*)$ is isomorphically identified as a subspace of the dual space of $T^{p}_{q}(\real^{d+1}_+; X)$. Moreover, it is norming for $T^{p}_{q}(\real^{d+1}_+; X)$ in the following sense, 
    \begin{align}\label{dual1}
        \| f \|_{ T^p_q(X) } \lesssim \max\lk{ p^{\frac{1}{q}}, p'^{\frac{1}{q'}} }\sup_g \jdz{ \int_{\real^{d+1}_+} \lara{ f(x, t), g(x, t) }_{X \times X^*} \,\frac{\d x\d t}{t} }, \quad 1 < p < \infty,
    \end{align}
    where the supremum is taken over all {$g\in C_c(\real^{d+1}_+)\otimes X^*$ such that $\|\mathcal{A}_{q'}(g)\|_{p'} \leq 1$};  and similarly,
    \begin{align}\label{dual2}
        \| f \|_{ T^p_q(X) } \lesssim \sk{\frac{p(q-1)}{q-p}}^{\frac{1}{q'}}\sup_g \jdz{ \int_{\real^{d+1}_+} \lara{ f(x, t), g(x, t) }_{X \times X^*} \,\frac{\d x\d t}{t} }, \quad 1 \leq p < q,
    \end{align}
    where the supremum is taken over all {$g\in C_c(\real^{d+1}_+)\otimes X^*$ such that $\|\mathcal{C}_{q'}(g)\|_{p'} \leq 1$}.
    Furthermore, if $X^*$ has the Radon-Nikod{\'y}m property, then
    \[
        T^{p'}_{q'}(\real^{d+1}_+;X^*) = \sk{ T^{p}_{q}(\real^{d+1}_+;X) }^*, \quad 1\leq p < \infty. 
    \]
\end{lemma}

\begin{proof}
We adopt the maps $i$ and $N$ used in the proof of Lemma \ref{Intpo}.  We first prove the estimate \eqref{dual1}. 

For any { $g \in C_c(\real^{d+1}_+) \ot X^*$ } and  $f \in T^p_q(\real^{d+1}_+; X)$, we denote by
\[
    g(f) = \int_{\real^{d+1}_+} \lara{ f(x, t), g(x, t) }_{X \times X^*} \,\frac{\d x\d t}{t}.
\]
Thus we have
\begin{equation}\label{g(f)}
\begin{split}
    |g(f)| & = \jdz{ \int_{\real^{d+1}_+} \lara{ f(y, t), g(y, t)\sk{ w_d^{-1}\int_{|x - y| < t} 1\,\d x } }_{X \times X^*} \,\frac{\d y\d t}{t^{d+1}} } \\
    & = w_d^{-1}\jdz{ \int_{\real^d}\int_{\real^{d+1}_+} \lara{ i(f)(x, y, t), i(g)(x, y, t) }_{X \times X^*} \,\frac{\d y\d t}{t^{d+1}}\d x }. 
\end{split}
\end{equation}
Since $i(f) \in \wt{T}^p_q$, we have $N(i(f)) = i(f)$. Then we deduce that 
\begin{equation}\label{normingf}
\begin{split}
    \| f \|_{T^p_q(X)} & = \| i(f) \|_{L^p(E)} = \sup_{g}\jdz{ \int_{\real^d}\int_{\real^{d+1}_+} \lara{ i(f)(x, y, t), g(x, y, t) }_{X \times X^*} \,\frac{\d y\d t}{t^{d+1}}\d x } \\
    & = \sup_{g}\jdz{ \int_{\real^d}\int_{\real^{d+1}_+} \lara{ N(i(f))(x, y, t), g(x, y, t) }_{X \times X^*} \,\frac{\d y\d t}{t^{d+1}}\d x } \\
    & = \sup_{g}\jdz{ \int_{\real^d}\int_{\real^{d+1}_+} \lara{ i(f)(x, y, t), N(g)(x, y, t) }_{X \times X^*} \,\frac{\d y\d t}{t^{d+1}}\d x }
\end{split}
\end{equation}
where the supremum is taken over all $g$ in the unit ball of $L^{p'}(\real^d; F)$. Notice that $N(g) = i\sk{i^{-1}(N(g))}$ and by \eqref{esN}
\[
    \| i^{-1}(N(g)) \|_{T^{p'}_{q'}(X^*) } = \| N(g) \|_{L^{p'}(F)} \lesssim \max\lk{ p^{\frac{1}{q}}, {p'}^{\frac{1}{q'}} }\| g \|_{L^{p'}(F)}.
\]
Consequently, combining \eqref{g(f)} and \eqref{normingf}, we obtain 
\[
    \| f \|_{T^p_q(X)} \lesssim \max\lk{ p^{\frac{1}{q}}, {p'}^{\frac{1}{q'}} } \sup_{g} |g(f)|, \quad 1 < p < \infty,
\]
where the supremum is taken over all $g\in C_c(\real^{d+1}_+)\otimes X^*$ such that $\|\mathcal{A}_{q'}(g)\|_{p'} \leq 1$, and we actually exploit a limiting argument: since not only the subset of $C_c(\real^{d+1}_+) \otimes X^*$ with norm $\|\mathcal{A}_{q'}(g)\|_{p'} \leq 1$ is contained in the unit ball of $T^{p'}_{q'}(\real^{d+1}_+; X^*)$, but also its closure contains the unit sphere, and thus one concludes that this subset is still norming for $T^p_q(\real^{d+1}_+; X)$. 

\smallskip

Now we deal with the estimate \eqref{dual2} in the case $1 < p < q$. Let $g\in L^{p'}(\real^d; F)$. By definition we have
\[
\begin{aligned}
    \| i^{-1}(N(g))(y, t) \|_{X^*}^{q'} & \leq \sk{ \frac{1}{|B(y, t)|} \int_{B(y, t)}\| g(z, y, t) \|_{X^*} \,\d z }^{q'} \\
    & \leq \frac{1}{|B(y, t)|} \int_{|z - y|< t} \| g(z, y, t) \|_{X^*}^{q'} \,\d z.
\end{aligned}
\]
For a ball $B$ in $\real^d$, we observe
\[
\begin{aligned}
    \int_{\wh{B}}\| i^{-1}(N(g))(y, t) \|_{X^*}^{q'}\,\frac{\d y\d t}{t} & \lesssim \int_{\wh{B}} \int_{|z - y|< t }\| g(z, y, t) \|_{X^*}^{q'} \,\d z\frac{\d y\d t}{t^{d+1}} \\
    & \leq \int_{2B} \int_{\real^{d+1}_+} \| \1_{\wh{B}}(y, t)g(z, y, t) \|_{X^*}^{q'} \,\frac{\d y\d t}{t^{d+1}}\d z\\
    & = \int_{2B} H^{q'}(z)\,\d z,
\end{aligned}
\]
where
\[
    H(z) = \sk{ \int_{\real^{d+1}_+} \| \1_{\wh{B}}(y, t)g(z, y, t) \|_{X^*}^{q'} \,\frac{\d y\d t}{t^{d+1}} }^{\frac{1}{q'}}.
\]
Then we have
\[
    \mathcal{C}_{q'}\mk{ i^{-1}(N(g)) }(x) \lesssim \sk{ \M(H^{q'})(x) }^{\frac{1}{q'}}, 
\]
where $\M$ is the Hardy-Littlewood maximal operator. Therefore when $q' < p' \leq \infty$, we obtain
\[
    \| \mathcal{C}_{q'}\mk{ i^{-1}(N(g)) } \|_{p'} \lesssim \| \M(H^{q'})^{\frac{1}{q'}} \|_{p'} \lesssim \sk{\frac{p(q-1)}{q-p}}^{\frac{1}{q'}}\| H \|_{p'} \leq \sk{\frac{p(q-1)}{q-p}}^{\frac{1}{q'}}\| g \|_{L^{p'}(F)}.
\]
Thus we observe 
\[
    \| f \|_{T^p_q(X)} \lesssim \sk{\frac{p(q-1)}{q-p}}^{\frac{1}{q'}}\sup_{g}|g(f)|, \quad 1 < p < q,
\]
with the supremum being taken over all $g\in C_c(\real^{d+1}_+)\otimes X^*$ such that $\|\mathcal{C}_{q'}(g)\|_{p'} \leq 1$.

For the endpoint case $p=1$ of \eqref{dual2}, because of the failure of vector-valued Calder\'{o}n-Zygmund theory, the above arguments adapted from \cite[Theorem 2.4]{Harboure1991} do not work any more. Instead, by using {the atomic decomposition of $T^1_q(\real^{d+1}_+; X)$}---Lemma \ref{atmicDectent}, one may carry out the classical arguments as in \cite[Theorem 1]{Coifman1985} in the present vector-valued setting, and we leave the details to the interested reader. 

\smallskip

When the Banach space $X^*$ has the Radon-Nikod{\'y}m property, one gets $F = E^*$ (cf. \cite[Theorem 1.3.10]{Hytoenen2016}). Then the duality follows from then an analogous argument in \cite{Harboure1991} for $1 < p < \infty$. Again, the duality in the case $p=1$ can be deduced as in the scalar-valued case  \cite[Theorem 1]{Coifman1985}, and we leave the details to the interested reader.
\end{proof}

\subsection{The two linear operators $\mathcal K$ and $\pi_L$} Let $\K: \real^{d+1}_+ \times \real^{d+1}_+ \to \real$ be a reasonable real-valued function such that for $f\in C_c(\real^{d+1}_+)\otimes X  $, the linear operator $\K$ is well defined as below,
\[
    \K(f)(x, t) := \int_{\real^{d+1}_+} \K_{t, s}(x, y)f(y, s)\,\frac{\d y\d s}{s}.
\]

%and by
%\[
%    \K_{t, s}(f)(x, s) =\int_{\real^d} \K_{t, s}(x, y)f(y, s) \,\d y.
%\]

\begin{lemma}\label{BoundTent}
Let $X$ be any fixed Banach space and $1<q<\infty$. Assume that the kernel $\K_{t, s}(x, y)$ satisfies the following estimation: there exist positive constants $ \kappa $, $ \e$, $C$ such that 
\begin{equation}\label{esKts}
    |\K_{t, s}(x, y)| \leq \frac{ C\min\lk{\frac{s}{t},  \frac{t}{s}}^\e \min\lk{ \frac{1}{t}, \frac{1}{s} }^d }{\sk{ 1 + \min\lk{ \frac{1}{t}, \frac{1}{s} }|x - y| }^{d + \kappa}  }.
\end{equation}
    Then the linear operator $\K$ initially defined on $C_c(\real^{d+1}_+)\otimes X$ extends to a bounded linear operator on $T^p_q(\real^{d+1}_+; X)$ for $1 \leq p < \infty$. More precisely, 
    \[
        \| \K(f) \|_{T^p_q(X)} \lesssim_{\e, \kappa} p^{\frac{1}{q}} \| f \|_{T^p_q(X)}, \quad \forall\, f \in T^p_q(\real^{d+1}_+; X), \ 1 \leq p < \infty.
    \]
    Furthermore, for any $f \in C_c(\real^{d+1}_+)\otimes X$, we have
    \[
        \| \mathcal{C}_q\sk{\K(f)} \|_p \lesssim_{\e, \kappa} \| \mathcal{C}_{q}(f) \|_{p}, \quad 1 \leq p \leq \infty.
    \]
\end{lemma}
    %\[
        %\| \mathcal{C}_q\sk{S(f)} \|_p \lesssim_{d, \e, \kappa} \sk{\frac{p}{p-q}}^{\frac{1}{q}}\| f \|_{T^p_q(X)}, \quad  \forall\, f \in T^p_q(X), \quad q < p \leq \infty.
    %\]
\begin{proof}
Fix $f \in C_c(\real^{d+1}_+)\otimes X$. Without loss of generality, we can assume $\kappa < \e$ from \eqref{esKts}.
We first deal with the case $p = q$. By H\"{o}lder's inequality, we have
\[
\begin{aligned}
    \int_{\real^{d+1}_+}\|\K(f)(x, t)\|_X^q\, \frac{\d y \d t}{t} & = \int_{\real^{d+1}_+}\norm{  \int_{\real^{d+1}_+} \K_{t, s}(y, w)f(w, s) \,\frac{\d w\d s}{s} }_X^q \, \frac{\d y \d t}{t} \\
    & \leq \int_{\real^{d+1}_+} \sk{ \int_{\real^{d+1}_+} |\K_{t, s}(y, w)|\, \frac{\d w \d s}{s} }^{\frac{q}{q'}}\\
    &\qquad \cdot\sk{ \int_{\real^{d+1}_+} |\K_{t, s}(y, w)|\|f(w, s)\|_X^q \,\frac{\d w\d s}{s}} \, \frac{\d y \d t}{t}.
\end{aligned}       
\]
We obtain that
\[
\begin{aligned}
    \int_0^{\infty}\int_{\real^d} |\K_{t, s}(y, w)|\, \frac{\d w \d s}{s} & \leq  \int_0^t\int_{\real^d} \frac{ Cs^{\e}t^{-\e} t^{-d} }{\sk{ 1 + t^{-1}|y - w| }^{d + \kappa}  } \,\frac{\d w\d s}{s} \\
    & \qquad + \int_t^{\infty}\int_{\real^d} \frac{ Ct^{\e}s^{-\e} s^{-d} }{\sk{ 1 + s^{-1}|y - w| }^{d + \kappa}  } \,\frac{\d w\d s}{s}\\
    & \lesssim_{\e, \kappa} \int_{\real^d} \frac{ C t^{-d} }{\sk{ 1 + t^{-1}|y - w| }^{d + \kappa}  } \,\d w + \int_t^{\infty} t^{\e}s^{-\e-1} \,\d s \\
    & \lesssim_{\e, \kappa} 1.
\end{aligned}
\]
It is clear that in the assumption of $\K_{t, s}(y, w)$, $(w, s)$ plays the same role as $(y, t)$. Thus
\begin{equation}\label{eTqtoq}
    \|  \K(f) \|_{L^q(\real^{d+1}_+; X)}^q \lesssim_{\e, \kappa} \int_{\real^{d+1}_+} \|f(w, s)\|_X^q \, \frac{\d w \d s}{s} = \| f \|_{L^{q}(\real^{d+1}_+; X)}^q.
\end{equation}
Then the case $p = q$ is done since $\| f \|_{T^q_q(X)} \approx {\| f \|_{L^q(\real^{d+1}_+; X)} }$. Moreover, from the proof we observe that $\K$ is always bounded on $L^p(\real^{d+1}_+; X)$ for $1 \leq p \leq \infty$. 

For $1 \leq p < q$, by the interpolation---Lemma \ref{Intpo}, it suffices to show the case $p=1$. By the atomic decomposition---Lemma \ref{atmicDectent}, It suffices to show that 
    \begin{equation}\label{pi_La}
        \| \K(a) \|_{T^1_{q}(X)} \lesssim_{\e, \kappa} 1,
    \end{equation}
    where $a$ is an $(X,q)$-atom with $\supp\, a \subset \wh{B}$ and $B = B(c_B, r_B)$.
%If such an inequality holds, for $f \in T^1_q(\real^{d+1}_+; X)$, there exists a sequence of complex numbers $\{ \lambda_j \}_{j \geq 1}$ and $(X, q)$-atoms $a_j(x, t)$ such that
%\[
%    f = \sum_{j \geq 1}\lambda_j a_j, \quad \|f\|_{T^1_{q}(X)} \approx \sum_{j \geq 1} |\lambda_j|.
%\]
%Then 
%\[
%    \| \K(f) \|_{T^1_q(X)} \leq \sum_{j =1}^{\infty} |\lambda_j| \| \K(a) \|_{T^1_q(X)} \lesssim_{d, \e, \kappa} \sum_{j =1}^{\infty} |\lambda_j| \approx \| f \|_{T^1_q(X)}.
%\]
%Now we prove the inequality \eqref{pi_La}. %We denote by $4B$ the ball that has the same center as $B$ while has the radius $4r_B$. 
One can write
\[
\begin{aligned}
     \|  \A_q[\K(a)] \|_{1} & = \int_{4B} \A_q[\K(a)](x) \,\d x + \int_{(4B)^C} \A_q[\K(a)](x) \,\d x \\
     & = I + II.
\end{aligned}
\]
From \eqref{eTqtoq} we obtain
\begin{equation}\label{e3.13}
    \|  \A_q[\K(a)] \|_{q}^q \lesssim_{\e, \kappa} \int_{\real^{d+1}_+} \|a(w, s)\|_X^q \, \frac{\d w \d s}{s} \leq |B|^{1-q}.
\end{equation}
Then we can estimate the term $I$:
\begin{equation}\label{esI}
    I \leq |4B|^{\frac{1}{q'}}\|  \A_q[\K(a)] \|_{q} \lesssim_{\e, \kappa} 1.
\end{equation}
Now we handle the second term $II$.
By H\"{o}lder's inequality, we observe
\[
\begin{aligned}
    \sk{ \A_q[\K(a)](x) }^q& \leq \int_0^{\infty}\int_{|y - x| < t} \sk{ \int_{\wh{B}} |\K_{t, s}(y, w)|^{q'}\, \frac{\d w \d s}{s} }^{\frac{q}{q'}}\cdot\sk{ \int_{\wh{B}} \|a(w, s)\|_X^q \,\frac{\d w\d s}{s}} \, \frac{\d y \d t}{t^{d+1}} \\
    & \lesssim |B|^{1-q}\int_0^{\infty}\int_{|y - x| < t} \sk{ \int_{\wh{B}}|\K_{t, s}(y, w)|^{q'}\, \frac{\d w\d s}{s} }^{\frac{q}{q'}}\, \frac{\d y \d t}{t^{d+1}} \\
    & = |B|^{1-q}\int_0^{r_B}\int_{|y - x| < t} \sk{ \int_{\wh{B}} |\K_{t, s}(y, w)|^{q'}\, \frac{\d w\d s}{s} }^{\frac{q}{q'}}\, \frac{\d y \d t}{t^{d+1}} \\
    & \qquad + |B|^{1-q}\int_{r_B}^{\infty}\int_{|y - x| < t} \sk{ \int_{\wh{B}} |\K_{t, s}(y, w)|^{q'}\, \frac{\d w\d s}{s} }^{\frac{q}{q'}}\, \frac{\d y \d t}{t^{d+1}} \\
    & =: J_1 + J_2.
\end{aligned}
\]
When $x \in (4B)^C,\ w \in B$, we have 
\[
    r_B < |x - w| \leq |x - y| + |y- w| < t+|y-w|,
\]
hence 
\[
    |x - c_B| \leq |x - w| +|w- c_B| < 2(t + |y-w|) \leq 2(\max\lk{t, s} + |y-w|).
\]
Therefore we observe from \eqref{esKts} that 
\[
\begin{aligned}
    |\K_{t, s}(y, w)| & \lesssim_{\e, \kappa} \frac{ \min\lk{ \frac{s}{t}, \frac{t}{s} }^{ \e } \min\lk{ \frac{1}{t}, \frac{1}{s} }^{d} }{ \sk{ \max\lk{t, s} + |y - w| }^{d + \kappa}\min\lk{\frac{1}{t}, \frac{1}{s}}^{d + \kappa } }\\
    & \lesssim_{\e, \kappa} \frac{ \min\lk{ \frac{s}{t}, \frac{t}{s} }^{ \e } \min\lk{ \frac{1}{t}, \frac{1}{s} }^{-\kappa} }{ |x - c_B|^{d + \kappa} } = \frac{ \min\lk{ s^\e t^{\kappa-\e}, t^\e s^{\kappa-\e} } }{ |x - c_B|^{d + \kappa} }.
\end{aligned}
\]
Then
    \[
    \begin{aligned}
        J_1 & \lesssim_{\e, \kappa} \frac{|B|^{1-q}|B|^{\frac{q}{q'}}}{|x - c_B|^{q(d + \kappa)}}\int_0^{r_B}\sk{ \int_0^{r_B}  \min\lk{ s^{q'\e} t^{q'(\kappa-\e)}, t^{q'\e} s^{q'(\kappa-\e)} } \, \frac{\d s}{s} }^{\frac{q}{q'}} \, \frac{ \d t}{t} \\
        & = \frac{ 1 }{|x - c_B|^{q(d + \kappa)} }\int_0^{r_B}\sk{ \int_0^{t} t^{q'(\kappa-\e)}s^{q'\e} \,\frac{\d s}{s} + \int_{t}^{r_B} s^{q'(\kappa- \e) }t^{q'\e} \,\frac{\d s}{s} }^{\frac{q}{q'}}\frac{\d t}{t}\\
        & \lesssim_{\e, \kappa} \frac{ r_B^{q\kappa} }{|x - c_B|^{q(d + \kappa)} }.
    \end{aligned}
    \]
    For $J_2$, since $t \geq  r_B \geq s$, and $|x - c_B| < 2(t + |y - w|)$,
    \[
    \begin{aligned}
        J_2 & \lesssim_{\e, \kappa} \frac{|B|^{1-q}|B|^{\frac{q}{q'}}  }{|x - c_B|^{q(d + \kappa)}} \int_{r_B}^{\infty} \sk{ \int_0^{r_B} s^{q' \e} t^{q'(\kappa-\e ) }  \, \frac{\d s}{s} }^{\frac{q}{q'}}\, \frac{ \d t}{t} \\
        & \lesssim_{\e, \kappa} \frac{ r_B^{q\kappa } }{|x - c_B|^{q(d + \kappa)} }.
    \end{aligned}
    \]
Thus
\[
    \A_q[\K(a)](x) \lesssim_{\e, \kappa} \frac{ r_B^{\kappa } }{|x - c_B|^{d + \kappa} }, \quad x \in (4B)^C.
\]
Since 
\[
\begin{aligned}
    \int_{(4B)^C} \frac{ r_B^{ \kappa }}{ |x - c_B|^{d + \kappa} } \,\d x  & = \sum_{m = 2}^{\infty} \int_{ 2^{m+1}B \setminus 2^{m}B }\frac{ r_B^{ \kappa } }{ |x - c_B|^{d + \kappa } } \,\d x \leq \sum_{m = 2}^{\infty} \int_{ 2^{m+1}B } \frac{ r_B^{ \kappa } }{ 2^{m(d + \kappa) }r_B^{d + \kappa} } \,\d x \\ 
    & \lesssim \sum_{m = 2}^{\infty}  \frac{ (2^{m+1}r_B)^d }{ 2^{m(d +  \kappa) }r_B^{d } }  \lesssim \sum_{m = 2}^{\infty} 2^{-m\kappa}   \lesssim_{\kappa} 1 ,
\end{aligned}
\]
we obtain
\begin{equation}\label{esII}
        II  \lesssim_{\e, \kappa} 1.
\end{equation}

%Combine \eqref{esI} and \eqref{esII}, we obtain that $\K$ extends to a bounded linear operator on $T^1_q(\real^{d+1}_+; X)$. 

For the case $q < p < + \infty$, we denote by $\K^*$ the adjoint operator. It is clear that the kernel of $\K^*$ has the same estimation as that of $\K$. For $f\in C_c(\real^{d+1}_+)\otimes X$, we obtain from Lemma \ref{dualtent} that
\[
\begin{aligned}
    \| \K(f) \|_{T^p_q(X)} & \lesssim p^{\frac{1}{q}}\sup_{g} \jdz{ \int_{\real^{d+1}_+} \lara{ \K(f)(x, t), g(x, t) }_{X \times X^*} \,\frac{\d x\d t}{t} } \\
    & = p^{\frac{1}{q}}\sup_{g} \jdz{ \int_{\real^{d+1}_+} \lara{ f(x, t), \K^*(g)(x, t) }_{X \times X^*} \,\frac{\d x\d t}{t} } \\
    & \leq p^{\frac{1}{q}}\| f \|_{T^p_q(X)} \| \K^*(g) \|_{T^{p'}_{q'}(X^*)} \lesssim p^{\frac{1}{q}} \| f \|_{T^p_q(X)},
\end{aligned}
\]
where the supremum is taken over all $g \in C_c(\real^{d+1}_+)\otimes X$ such that $\| \A_{q'}(g) \|_{p'} \leq 1$. Consequently, we observe that $\K$ extends to a bounded linear operator on $T^p_q(\real^{d+1}_+; X)$ for $1 \leq p < \infty$. More precisely,
\[
    \| \K(f) \|_{T^p_q(X)} \lesssim_{\e, \kappa} p^{\frac{1}{q}}\| f \|_{T^p_q(X)}, \quad \forall\, f \in T^p_q(\real^{d+1}_+; X).
\]

Now we prove the second assertion of this lemma. Fix $f \in C_c(\real^{d+1}_+)\otimes X$, take a ball $B$ in $\real^d$, we can write
\[
\begin{aligned}
    \sk{\int_{\wh{B}} \| \K(f)(x, t) \|_X^q \,\frac{\d x\d t}{t}}^{\frac{1}{q}} & = \sup_{g} \jdz{\int_{\wh{B}} \lara{ \K(f)(x, t), g(x, t) }_{X \times X^*} \,\frac{\d x\d t}{t} }\\
    & = \sup_{g}\jdz{ \int_{\wh{B}} \lara{ f(x, t), \K^*(g)(x, t) }_{X \times X^*} \,\frac{\d x\d t}{t}} \\
    & \leq \sup_g \| \K^*(g%\1_{\wh{B}}
) \|_{L^{q'}(\wh{B}; X^*)} \sk{ \int_{\wh{B}} \| f(x, t) \|_X^q \,\frac{\d x\d t}{t} }^{\frac{1}{q}},
\end{aligned}
\]
where the supremum is taken over all $g$ in the unit ball of $L^{q'}(\wh{B}; X^*)$. From \eqref{eTqtoq} we know that
\[
    \|\K^*(g) \|_{L^{q'}(\real^{d+1}_+; X^*)} \lesssim_{\e, \kappa} \| g \|_{L^{q'}(\wh{B}; X^*)}. 
\]
Thus for any $x \in \real^d$,
\[
\begin{aligned}
    \sk{ \mathcal{C}_q\mk{\K(f)}(x) }^q & = \sup_{x \in B} \frac{1}{|B|} \int_{\wh{B}} \| \K(f)(x, t) \|_X^q \,\frac{\d x\d t}{t} \\
    & \lesssim_{\e, \kappa} \sup_{x \in B} \frac{1}{|B|}\int_{\wh{B}} \| f(x, t) \|_X^q \,\frac{\d x\d t}{t} = \sk{ \mathcal{C}_q(f)(x) }^q.
\end{aligned}
\]
Therefore we obtain
\[
    \| \mathcal{C}_q\mk{\K(f)} \|_p \lesssim_{\e, \kappa} \| \mathcal{C}_q\sk{f} \|_p, \quad 1 \leq p \leq \infty.
\]
Moreover, from \eqref{A-C} we also observe
\[
    \| \mathcal{C}_{q}\mk{\K(f)} \|_{p}  \lesssim_{\e, \kappa} \sk{\frac{p}{p-q}}^{\frac{1}{q}} \| \A_{q}(f) \|_{p}, \quad q < p \leq \infty.
\]
The proof is completed.
\end{proof}

Now we come to the second important linear operator, which will relate the tent space $T^p_q(\real^{d+1}_+; X)$ to the Hardy space $H^p_{q, L}(\real^d; X)$.

%The operator $L$ is assumed to satisfy \eqref{pt1} without further explanation. 

Recall the operator $\Q(f)(x, t) = -tLe^{-tL}(f)(x)$. Define the operator $\pi_L$ acting on $C_c(\real^{d+1}_+)\otimes X$ as
\[
    \pi_L(f)(x) = \int_{0}^{\infty} \Q (f(\cdot, t))(x, t) \,\frac{\d t}{t}, \quad \forall\, x \in \real^d.
\]
It is easy to verify that $\pi_L$ is well-defined. 
%If we take $Q_t^*$ in place of $Q_t$, we denote the operator by $\pi_L^*$. 
The following lemma asserts that $\pi_L$ extends to a bounded linear operator from $T^p_q(\real^{d+1}_+; X)$ to $H^p_{q, L}(\real^d; X)$. We will denote it by $\pi_L$ as well. 

\begin{lemma}\label{l3.8}
Let $X$ be any fixed Banach space and $1<q<\infty$. The operator $\pi_L$ initially defined on $C_c(\real^{d+1}_+)\otimes X$ extends to a bounded linear operator from $T^p_q(\real^{d+1}_+; X)$ to $H^p_{q, L}(\real^d; X)$ for $1 \leq p < \infty$. More precisely,
    \[
        \| \pi_L(f) \|_{H^p_{q, L}(X)} \lesssim_{\beta} p^{\frac{1}{q}}\| f \|_{T^p_q(X)}, \quad \forall\, f \in T^p_q(\real^{d+1}_+;X), \ 1 \leq p < \infty.
    \]
    Furthermore, for any $f \in C_c(\real^{d+1}_+)\otimes X$, we have
    \[
        \| \pi_L(f) \|_{BMO^p_{q, L}(X)} \lesssim_{\beta} \| \mathcal{C}_q(f) \|_{p}, \quad 1 \leq p \leq \infty.
    \]
\end{lemma}

\begin{proof}
Let $f\in C_c(\real^{d+1}_+)\otimes X$. Recall that $k(t, x, y)$ is the kernel of the operator $\Q$, then
\begin{equation}
\begin{split}
    \Q[\pi_L(f)](x, t) & = \int_{\real^d} k(t, y, z)\pi_L(f)(z) \,\d z  \\
    & = \int_{\real^d} k(t, y, z) \sk{ \int_{\real^{d+1}_+} k(s, z, w)f(w, s) \,\frac{\d w\d s}{s} } \,\d z  \\
    & = \int_{\real^{d+1}_+} \sk{ \int_{\real^d} k(t, y, z) k(s, z, w) \, \d z }f(w, s) \,\frac{\d w\d s}{s}.
\end{split}
\end{equation}
We denote by
\[
    \Phi_{t, s}(y, w) = \int_{\real^d} k(t, y, z) k(s, z, w) \, \d z. 
\]
Note that $k(t, \cdot, \cdot)$ is the kernel of the operator $\Q=-te^{-tL}$, thus $\Phi_{t,s}(\cdot,\cdot)$ is the kernel of $-tLe^{-tL}\circ(-sLe^{-sL})=tsL^2e^{-(t+s)L}$. On the other hand, $\partial^2_r (e^{-rL})|_{r=t+s}=L^2e^{-(t+s)L}$ which has the kernel $\partial^2_r K(r,\cdot,\cdot)|_{r=t+s}$. Then by \eqref{derpt}, we obtain 
\begin{align*}
    |\Phi_{t, s}(y, w)| \lesssim_{d, \beta} \frac{ts}{(t+s)^{2-\beta}(t+s+|y - w|)^{d+ \beta} }  \lesssim_{\beta} \frac{ \min\lk{\frac{s}{t},  \frac{t}{s}} \min\lk{ \frac{1}{t}, \frac{1}{s} }^d }{\sk{ 1 + \min\lk{ \frac{1}{t}, \frac{1}{s} }|x - y| }^{d + \beta}  }.
\end{align*}
Denote by 
\begin{align}\label{Proj}
    \mathcal{P} = 4 \mathcal Q \circ \pi_L. 
\end{align}
 From Lemma \ref{BoundTent}, we conclude that $\mathcal{P}$ initially defined on $C_c(\real^{d+1}_+)\otimes X$ extends to a bounded linear operator on $T^p_q(\real^{d+1}_+; X)$. Moreover,
\[
    \norm{  \mathcal{P}(f) }_{T^p_q(X)} \lesssim_{\beta} p^{\frac{1}{q}}\norm{ f}_{T^p_q(X)}.
\]
Therefore
\[
    \| \pi_L(f) \|_{H^p_{q, L}} = 4^{-1}\| \mathcal{P}(f) \|_{T^p_{q}(X)} \lesssim_{\beta} p^{\frac{1}{q}}\| 
    f \|_{T^p_q(X)},
\]
which is the desired assertion.

For the second part, we obtain the desired assertion from Lemma \ref{BoundTent} immediately.
\end{proof}

\begin{remark}\label{extend:PiL}
    One can verify that $\mathcal{P} \circ \mathcal{P} = \mathcal{P}$, thus $\mathcal{P}$ serves as a continuous projection from $T^p_q\sk{\real^{d+1}_+; X}$ onto itself. Indeed, we can also obtain this lemma under the assumption that $L$ is a sectotrial operator satisfying only \eqref{pt1}.
\end{remark}

\section{vector-valued intrinsic square functions}
In this section, we begin with the introduction of vector-valued intrinsic square functions, originally presented by Wilson in \cite{Wilson2008} in the case of convolution operators. We then proceed to compare them with the $q$-variant of Lusin area integral associated with a generator $L$.%{mainassumption}

Recall that $L$ is assumed to be a sectorial operator of type $\alpha$ ($0 \leq \alpha < \pi/2$) satisfying assumptions \eqref{pt10}, \eqref{pt2} and \eqref{pt3} with $ \beta > 0, 0 < \gamma \leq 1$. Define $ \mathcal{H}_{\gamma,\beta} $ as the family of functions $ \f : \real^d \times \real^d \to \real $ such that
\begin{equation}\label{Nicef1}
  |\f(x,y)| \leq \frac{1}{(1 + |x-y|)^{d+\beta}},
\end{equation}
\begin{equation}\label{Nicef2}
  | \f(x+h,y) - \f(x,y) | + | \f(x,y+h) - \f(x,y) | \leq \frac{|h|^\gamma  }{(1 + |x-y|)^{d+\beta+\gamma}}
\end{equation}
whenever $ 2|h| \leq 1 + |x-y| $ and
\begin{equation}\label{Nicef3}
  \int_{\real^d} \f(x,y)\,\d x = \int_{\real^d} \f(x,y)\,\d y = 0.
\end{equation}
For $\f\in \mathcal{H}_{\gamma,\beta} $, define $ \f_t(x,y) = t^{-d}\f(t^{-1}x,t^{-1}y) $.

Let $ f \in C_c(\real^d) \ot X$. We define
\[
    A_{\gamma,\beta} (f)(x,t) = \sup_{\f \in \mathcal{H}_{\gamma,\beta} } \norm{ \int_{\real^d} \f_t(x,y)f(y)\,\d y }_X, \quad \forall\, (x, t) \in \real^{d + 1}_+.
\]
Then the intrinsic square functions of $f$ are defined as
\[
S_{q,\gamma,\beta}(f)(x) = \sk{ \int_{\Gamma(x)} \sk{A_{\gamma,\beta}(f)(y,t)}^q \,\frac{\d y\d t}{t^{d+1}}  }^{\frac{1}{q}},
\]
and
\[
G_{q,\gamma,\beta}(f)(x) = \sk{ \int_{0}^{\infty} \sk{A_{\gamma,\beta}(f)(x,t)}^q \,\frac{\d t}{t}  }^{\frac{1}{q}}.
\]
%\[
%\sigma_{q,\gamma,\beta}(f)(x) := \sk{ \sum_{k=-\infty}^{\infty} \sk{A_{q,\gamma,\beta}(f)(x,2^k)}^q  }^{\frac{1}{q}}.
%\]

%%%%
%%%%

\begin{thm}\label{equiv}
Let $X$ be any fixed Banach space, $1<q<\infty$ and $1\leq p<\infty$. Let $L$ be any fixed sectorial operator $L$ satisfying \eqref{pt10}, \eqref{pt2} and \eqref{pt3}. For any $f\in C_c(\mathbb R^d)\otimes X$, we have 
 \begin{align}\label{SG2}
  S_{q, \gamma,\beta}(f)(x) \approx_{\gamma, \beta} G_{q, \gamma,\beta}(f)(x),
 \end{align}

\begin{align}\label{Square Wilson}
    S_{q, L}(f)(x) \lesssim  S_{q, \gamma,\beta}(f)(x), \quad G_{q, L}(f)(x) \lesssim  G_{q, \gamma,\beta}(f)(x),
\end{align}
 and 
   \begin{align}\label{Wilson Square}
\| S_{q, \gamma,\beta}(f) \|_p \lesssim_{\gamma,\beta} p^{\frac{1}{q}}\| S_{q,L}(f) \|_{p}.
  \end{align}
\end{thm}

\begin{remark}
The following $g$-function version of \eqref{Wilson Square} holds also
   \begin{align}\label{Wilson G}
\| G_{q, \gamma,\beta}(f) \|_p \lesssim_{\gamma,\beta} {p^{\frac{2}{q}}  }\| G_{q,L}(f) \|_{p}.
  \end{align}
But its proof is much more involved and depends in turn on Theorem \ref{main2} that will be concluded in the next section. %On the other hand, {\color{red} It seems obtained from a weighted inequality similar to \cite[Section 5]{Xu2021}}.
\end{remark}

As in the classical case \cite{Wilson2007}, the assertions \eqref{SG2} and \eqref{Square Wilson} can be deduced easily from the following facts on $ \mathcal{H}_{\gamma,\beta}$. %However due to the lack of the translation and dilation structures in the present setting, these facts are needed to be justified. 

\begin{lemma}\label{lem:facts}
Let $\f \in \mathcal{H}_{\gamma, \beta}$. The following properties hold:

{\rm (i)} if $t \geq 1$, then $t^{-d-\gamma}\f_t \in \mathcal{H}_{\gamma, \beta}$;

{\rm(ii)} if $|z| \leq 1$, $t \geq 1$, then $ (2t)^{-d-\gamma-\beta} \sk{ \f^{(z)} }_{t} \in \mathcal{H}_{\gamma, \beta} $, where $ \f^{(z)}(x,y) = \f(x-z,y) $.
\end{lemma}

\begin{proof}
The proof is similar to the case of Wilson \cite{Wilson2007}, while the present setting is non-convolutive, let us  give the sketch. The claim (i) is easy by definition. For the claim (ii), notice that
\[
    2^{-1} (1 + |x-y|) \leq  1 + |(x-z) - y| \leq 2(1 + |x- y|). 
\]
By definition, we have
\[
    |\f^{(z)}(x, y)| = | \f(x-z, y) | \leq \frac{1}{(1 + |(x-z) - y |)^{d+\beta} } \leq  \frac{2^{d+\beta}}{(1 + |x- y|)^{d+\beta}}.
\]
and 
\[
\begin{aligned}
    |\f^{(z)}(x + h, y) - \f^{(z)}(x, y)| & = | \f(x-z + h, y) - \f(x-z, y) | \\
    & \leq \frac{|h|^{\gamma}}{(1 + |(x-z) - y |)^{d+\beta + \gamma} } \\
    & \leq  \frac{2^{d+\beta + \gamma}|h|^{\gamma}}{(1 + |x- y|)^{d+\beta+\gamma}}.
\end{aligned}
\]
The same H\"{o}lder continuity estimation holds for the variable $y$. Thus we obtain $2^{-d-\beta-\gamma}\f^{(z)} \in \mathcal{H}_{\gamma, \beta}$. Then the claim (ii) follows from the claim (i). 
\end{proof}

With Lemma \ref{lem:facts}, the assertions \eqref{SG2} and \eqref{Square Wilson} will follow easily.
The most challenging part of Theorem \ref{SG2} lies in \eqref{Wilson Square}. In addition to the interpolation and duality theory on the (vector-valued) tent space that have been built in Section \ref{Tentspace}, the following pointwise estimate is another technical part in the proof of estimate \eqref{Wilson Square}.

Recall that $k(t, x, y)$ is the kernel of the operator $\Q$.  Let ${\theta \in \mathcal{H}_{\gamma, \beta}}$, define
 \[
        \L^\theta_{t, s}(y, w) = \int_{\real^d}\theta_t(y, z)k(s, z, w)\,\d z.
    \]

\begin{lemma}\label{lem:thetakestimates}
Let $\nu = 2^{-1}\min\lk{ \gamma, \beta }$ and $ \zeta = (d + 2^{-1}\beta )(d + \beta)^{-1}$, then
   \[
        \sup_{ \theta \in \mathcal{H}_{\gamma, \beta} }|\L^\theta_{t, s}(y, w)| \lesssim_{\gamma, \beta}  \frac{ \min\lk{ \frac{s}{t}, \frac{t}{s} }^{ (1 - \zeta)\nu } \min\lk{ \frac{1}{t}, \frac{1}{s} }^d }{ \sk{ 1 + \min\lk{ \frac{1}{t}, \frac{1}{s} }|y - w|}^{d + \frac{1}{2}\beta} }. 
    \]    
\end{lemma}

\begin{proof}
To estimate the kernel $\L^\theta_{t, s}(y, w)$, we follow a similar argument presented in \cite[Chapter 8, 8.6.3]{Grafakos2009}. 

Let $\theta \in \mathcal{H}_{\gamma, \beta}$, we have
    \begin{equation}\label{theta1}
        |\theta_t(y, z)| \leq \frac{ t^{-d}}{(1 + t^{-1}|y - z|)^{d + \beta}}, \quad \forall \, y, z \in \real^d,\ t> 0.
    \end{equation}
    For $ 2|z - z'| < t + |y- z|$, we have 
    $t^{-1}|z - z'| < 1 + t^{-1}|y - z|$, then
    \[
    \begin{aligned}
        |\theta_t(y, z) - \theta_t(y, z')| & \leq \frac{ t^{-d-\gamma}|z - z'|^\gamma}{(1+ t^{-1}|y - z|)^{d + \beta + \gamma}} \leq \frac{\min\lk{ (t^{-1}|z - z'|)^\gamma, \sk{1 + t^{-1}|y - z|}^\gamma } }{t^d(1+ t^{-1}|y - z|)^{d + \beta + \gamma}} \\
        & \lesssim \frac{ \min\lk{ 1, (t^{-1}|z-z'|)^{\gamma} } }{t^d}.
    \end{aligned}
    \]    
    For $ 2|z - z'| \geq t + |y- z|$, we have $ t^{-1}|z - z'| \geq 1/2$, then
    \[
        |\theta_t(y, z) - \theta_t(y, z')| \leq |\theta_t(y, z)| + |\theta_t(y, z')| \leq 2 t^{-d}\lesssim \frac{ \min\lk{ 1, (t^{-1}|z-z'|)^{\gamma} } }{t^d}.
    \]
    Hence
    \[
        |\theta_t(y, z) - \theta_t(y, z')| \lesssim \frac{ \min\lk{ 1, (t^{-1}|z-z'|)^{\gamma} } }{t^d}, \quad\forall\, y, z, z' \in \real^d,\ t > 0.
    \]
    On the other hand, Lemma \ref{2.8} asserts that there exists a positive constant $C_k$ such that $ C_k^{-1}(k(s, \cdot, \cdot))_{s^{-1}} \in \mathcal{H}_{\gamma, \beta} $ (see also the {Convention} afterwards). Thus, one gets for all $ w, z, z' \in \real^d,\ s > 0$,
    \[
        |k(s, z, w)| \lesssim \frac{C_k s^{-d}}{(1 + s^{-1}|z - w|)^{d + \beta}}, \quad |k(s, z, w) - k(s, z', w)| \lesssim \frac{ C_k\min\lk{ 1, (s^{-1}|z - z'|)^{\gamma} } }{s^d}.
    \]
    
    Now we start to deal with the kernel $\L^\theta_{t, s}(y, w)$. By symmetry, it suffices to handle the case $s \leq t$. First we observe the following estimate,
    \[
    \begin{aligned}
        \int_{\real^d} \frac{ s^{-d} \min\lk{ 1, (t^{-1}|u|)^{\gamma} } }{ (1 + s^{-1}|u|)^{d + \beta} }\, \d u & = \int_{|u| < t} \frac{ s^{-d}  (t^{-1}|u|)^{\gamma} }{ (1 + s^{-1}|u|)^{d + \beta} }\, \d u + \int_{|u| > t} \frac{ s^{-d} }{ (1 + s^{-1}|u|)^{d + \beta} }\, \d u \\
        & \leq \int_{|v| < t/s} \sk{\frac{s}{t}}^\gamma\frac{  |v|^{\gamma} }{ (1 + |v|)^{d + \beta} }\, \d v + \int_{|u| > t} s^{\beta}|u|^{-d-\beta}\, \d u \\
        & =: J_1 + J_2
    \end{aligned}
    \]
    Taking $\nu = 2^{-1}\min\lk{\gamma, \beta}$, and we have $|v|^\gamma < (t/s)^{\gamma - \nu}|v|^\nu $. Then we obtain
    \[
        J_1 \leq \sk{ \frac{s}{t} }^{\nu} \int_{\real^d} \frac{  |v|^{\nu} }{ (1 + |v|)^{d + \beta} }\, \d v \lesssim_{\gamma, \beta} \sk{ \frac{s}{t} }^{\nu}.
    \]
    For $J_2$, we have
    \[
        J_2 \lesssim \int_{t}^{\infty} s^\beta r^{-\beta - 1}\,\d r \lesssim_{\beta} \sk{ \frac{s}{t} }^{\beta} \leq \sk{ \frac{s}{t} }^{\nu}. 
    \]
    Thus for $s \leq t$, by the vanishing property \eqref{Nicef3}  of $k(s,\cdot,w)$, one gets
     \[
    \begin{aligned}
        |\L^\theta_{t, s}(y, w)| & \leq \jdz{ \int_{\real^d}\mk{\theta_t(y, z) - \theta_t(y, w)}k(s, z, w) \, \d z } \\
        & \leq C_k \int_{\real^d} \frac{ \min\lk{ 1, (t^{-1}|z-w|)^{\gamma} } }{t^d}\frac{s^{-d}}{ (1 + s^{-1}|z - w|) } \, \d z \\
        & \lesssim_{\gamma, \beta} t^{-d} \sk{ \frac{s}{t} }^{\nu } \leq \min\lk{ \frac{1}{t}, \frac{1}{s} }^d\min\lk{ \frac{s}{t}, \frac{t}{s} }^{ \nu }.
    \end{aligned}
    \]
  %  For $s \geq t$, we exchange $s$ and $t$, then we observe
  %  \[
   % \begin{aligned}
     %   |\L_{t, s}^\theta(y, w)| & \leq \jdz{ \int_{\real^d}\theta_t(y, z)\mk{k(s, z, w) - k(s, y, w)} \, \d z } \\ 
     %   & \leq C_k \int_{\real^d} \frac{t^{-d}}{ (1 + t^{-1}|z - y|) }\frac{ \min\lk{ 1, (s^{-1}|z-y|)^{\gamma} } }{s^d} \, \d z \\
       % & \lesssim_{d, \beta, \gamma} \min\lk{ \frac{1}{t}, \frac{1}{s} }^d\min\lk{ \frac{s}{t}, \frac{t}{s} }^{ \nu }.
  %  \end{aligned}
 %   \]
    On the other hand,
    \[
        |\L^\theta_{t, s}(y, w)| \leq  \int_{\real^d} |\theta_t(y, z) ||k(s, z, w) |\, \d z  \lesssim_{\beta} \frac{ \min\lk{ \frac{1}{t}, \frac{1}{s} }^d }{ \sk{ 1 + \min\lk{ \frac{1}{t}, \frac{1}{s} }|y - w|}^{d + \beta} }.
    \]
Let $ \zeta = (d + 2^{-1}\beta )(d + \beta)^{-1}$, we then get
    \[
        |\L^\theta_{t, s}(y, w)| = |\L^\theta_{t, s}(y, w)|^{1-\zeta}|\L^\theta_{t, s}(y, w)|^\zeta \lesssim_{\gamma, \beta} \frac{ \min\lk{ \frac{s}{t}, \frac{t}{s} }^{ (1 - \zeta)\nu } \min\lk{ \frac{1}{t}, \frac{1}{s} }^d }{ \sk{ 1 + \min\lk{ \frac{1}{t}, \frac{1}{s} }|y - w|}^{d + \frac{1}{2}\beta} }.
    \]
    It is clear that the estimation of $\L^\theta_{t, s}(y, w)$ is independent of the choice of $\theta$, and thus the desired estimate is obtained.
\end{proof}

Now let us  prove Theorem \ref{equiv}.

\begin{proof}
The pointwise estimate \eqref{SG2} follows from Lemma \ref{lem:facts} (ii). Indeed, for $|x - y|< t$, let $w = (x - y)/t$; then for any $\f \in \mathcal{H}_{\gamma, \beta}$, we have $2^{-d - \beta -\gamma}\f^{(w)} \in \mathcal{H}_{\gamma, \beta}$. Hence
\[
\begin{aligned}
    A_{\gamma, \beta}(f)(x, t) & = \sup_{\f \in \mathcal{H}_{\gamma,\beta} } \norm{ \int_{\real^d} \f_t(x,z)f(z)\,\d z }_X\\
    & \leq 2^{d + \beta + \gamma}\sup_{\f^{(w)} \in \mathcal{H}_{\gamma,\beta} } \norm{ \int_{\real^d} \sk{\f^{(w)}}_t(x,z)f(z)\,\d z }_X \\
    & = 2^{d + \beta + \gamma}A_{\gamma, \beta}(f)(y, t).
\end{aligned}
\]
Exchanging $x$ and $y$ and taking $-w$ in place of $w$, the reverse inequality is also true. Then \eqref{SG2} follows immediately.

Now we turn to the pointwise estimates \eqref{Square Wilson}.
Lemma \ref{2.8} asserts that there exists a positive constant $C_k$ such that $ C_k^{-1}(k(t, \cdot, \cdot))_{t^{-1}} \in \mathcal{H}_{\gamma, \beta} $ (see also the {Convention} afterwards). Consequently, for all $x\in \real^d, \ t>0$, we have
\begin{equation}\label{Qf-Af}
\begin{aligned}
    \| \Q(f)(x, t) \|_X & = \norm{ \int_{\real^d} k(t, x, y)f(y)\,\d y }_X = C_k\norm{ \int_{\real^d} \sk{ C_k^{-1}(k(t, x, y))_{t^{-1}} }_tf(y)\,\d y }_X\\
    & \leq C_k \sup_{\f \in \mathcal{H}_{\gamma, \beta} }\norm{ \int_{\real^d} \f_t(x, y)f(y)\,\d y }_X = C_k A_{\gamma, \beta}(f)(x, t).
\end{aligned}
\end{equation}
Then the estimates \eqref{Square Wilson} follows trivially.

Below we explain the proof of \eqref{Wilson Square}. Let $h \in C_c(\real^{d+1}_+)\otimes X$, we have
    \[
    \begin{aligned}
        A_{\gamma, \beta}[\pi_L(h)](y, t) & = \sup_{\theta \in \mathcal{H}_{\gamma, \beta}} \norm{ \int_{\real^d}\theta_t(y, z)\sk{ \int_{\real^{d+1}_+} k(s, z, w)h(w, s)\, \frac{\d w\d s}{s} } \,\d z }_X \\
        & = \sup_{\theta \in \mathcal{H}_{\gamma, \beta}} \norm{ \int_{\real^{d+1}_+} \sk{ \int_{\real^d}\theta_t(y, z)k(s, z, w)\,\d z } h(w, s)\, \frac{\d w\d s}{s} }_X \\
        & = \sup_{\theta \in \mathcal{H}_{\gamma, \beta}} \norm{ \int_{\real^{d+1}_+} \L^\theta_{t, s}(y, w) h(w, s)\, \frac{\d w\d s}{s} }_X \\
        & \leq \int_{\real^{d+1}_+} \sk{ \sup_{\theta \in \mathcal{H}_{\gamma, \beta}} |\L^\theta_{t, s}(y, w)| }\| h(w, s) \|_X\, \frac{\d w\d s}{s} \\
        & =: \L( \| h \|_X )(y, t),
    \end{aligned}
    \]
where the linear operator $\L$ has the kernel 
\[
    \L_{t,s}(y,w)=\sup_{\theta \in \mathcal{H}_{\gamma, \beta}}|\L^\theta_{t, s}(y, w)|.
\]
Then by Lemma \ref{lem:thetakestimates} and Lemma \ref{BoundTent} in the case $X = \com$, one obtaines
    \[
        \| \L( \| h \|_X ) \|_{T^p_q(\com)} \lesssim_{\gamma, \beta}p^{\frac{1}{q}} \| \|h\|_X \|_{T^p_q(\com)} =  p^{\frac{1}{q}}\| h \|_{T^p_q(X)},  \quad 1 \leq p < \infty.
    \]
    Therefore
    \[
        \| A_{\gamma, \beta}[\pi_L(h)] \|_{T^p_q(\com)} \lesssim_{\gamma, \beta} p^{\frac{1}{q}}\| h \|_{T^p_q(X)},  \quad 1 \leq p < \infty.
    \]
Let $f \in C_c(\mathbb R^d)\otimes X$, then we have $\Q(f) \in T^p_q(\real^{d+1}_+; X)$; {moreover from the formula \eqref{funcal1} and the fact that the fixed point subspace of $L^p(\mathbb R^d; X)$ is $0$ (see the statement before Remark \ref{rk:Lstar}), the following Calder\'on identity holds
\begin{equation}\label{funcal}
    f= 4 \int_{0}^{\infty} \Q\mk{\Q(f)(\cdot, t)}(\cdot, t) \,\frac{\d t}{t}.
\end{equation}}
Therefore, one has that for $1 \leq p < \infty$,
    \[
    \begin{aligned}
        \| S_{q, \gamma, \beta}(f) \|_p & = \| A_{\gamma, \beta}(f) \|_{T^p_q(\com)} = 4\| A_{\gamma, \beta}\mk{ \pi_L(\Q(f)) } \|_{T^p_q(\com)} \\
        & \lesssim_{\gamma, \beta} p^{\frac{1}{q}}\| \Q(f) \|_{T^p_q(X)} = p^{\frac{1}{q}}\| S_{q,L}(f) \|_p,
    \end{aligned}
    \]
    which is the desired inequality.
\end{proof}

{
\begin{remark}\label{equiv-BMO}
    For any $f \in C_c(\real^d) \ot X$, by Lemma \ref{BoundTent}, we also obtain
    \[
        \| A_{\gamma, \beta}(f) \|_{T^\infty_q(\com)} = 4\| A_{\gamma, \beta}\mk{ \pi_L(\Q(f)) }  \|_{T^\infty_q(\com)} \lesssim_{\gamma, \beta} \| \Q(f) \|_{T^\infty_q(X)} = \| f \|_{BMO_{q, L}(X)}.      
    \]
    Together with the pointwise estimate \eqref{Qf-Af}, one gets the BMO--version of Theorem \ref{main2}: Let $L$ be a generator as in Theorem \ref{main2}, then
    \begin{equation}
        \| f \|_{BMO_{q, L}(X)} \approx_{\gamma, \beta} \| f \|_{BMO_{q, \sqrt{\Delta}}(X)}.
    \end{equation}
\end{remark}
}

\section{Proof of the main Theorem}\label{Proof}
As pointed out in the introduction, the equivalence \eqref{SL=SD} in Theorem \ref{main2} is an easy consequence of Theorem \ref{equiv}; but for another equivalence  \eqref{SL=SG}, we need to develop fully Mei's duality arguments between vector-valued Hardy and BMO type spaces \cite{Mei2007}. This will be accomplished in the present section by combining the theory of vector-valued tent spaces and vector-valued Wilson's square functions---Theorem \ref{equiv}.

 First of all, based on the duality between tent spaces---Lemma \ref{dualtent},  the boundedness of the projection $\pi_L$---Lemma \ref{l3.8}---yields the following vector-valued Fefferman-Stein duality theorem.
\begin{thm}\label{Dual}
Let $X$ be any fixed Banach space and $1<q<\infty$. Let $L$ be as in Theorem \ref{main2}. Both the spaces $BMO_{q', L^*}^{p'}(\real^d; X^*)$ and $H^{p'}_{q', L^*}(\real^d; X^*)$ are isomorphically identified as subspaces of the dual space of $H^p_{q, L}(\real^d; X)$. Moreover, they are norming for $H^p_{q, L}(\real^d; X)$ in the following sense, 
    \[
        \| f \|_{ H^p_{q, L} } \lesssim_{\beta} \max\lk{ p^{\frac{1}{q}}p'^{\frac{1}{q'}}, p'^{\frac{2}{q'}} } \sup_g \jdz{ \int_{\real^d} \lara{ f(x), g(x) }_{X \times X^*} \,\d x }, \quad 1 < p < \infty,     
    \]
    where the supremum is taken over all $g \in C_c(\real^d) \ot X^*$ such that $\|g\|_{H^{p'}_{q',L^*}(X^*)}\leq 1$, and similarly,
    \[
        \| f \|_{ H^p_{q, L} } \lesssim_{\beta} \sk{\frac{p(q-1)}{q-p}}^{\frac{1}{q'}} \sup_g \jdz{ \int_{\real^d} \lara{ f(x), g(x) }_{X \times X^*} \,\d x }, \quad 1 \leq p < q,     
    \]
    where the supremum is taken over all $g \in C_c(\real^d) \ot X^*$ such that $\|g\|_{BMO^{p'}_{q',L^*}(X^*)}\leq 1$.    
    Furthermore, if $X^*$ has the Radon-Nikod{\'y}m property. Then 
    \[
    \begin{aligned}
        & BMO^{p'}_{q', L^*}(X^*) = \sk{ H^p_{q, L}(\real^d; X) }^*,\quad 1 \leq p < q ;\\
        & H^{p'}_{q', L^*}(X^*)= \sk{ H^p_{q, L}(\real^d; X)}^*,  \quad 1 < p < \infty.
    \end{aligned}
    \]
    %More precisely, 
    %\begin{enumerate}
    %    \item For $g \in BMO_{q', L^*}(\real^d; X^*)$, the linear functional $\ell$ given by
    %    \[
    %        \ell(f) = \int_{ \real^d } \lara{ f(x), g(x) }_{X \times X^*} \,\d x
    %    \]
    %    defined on the dense subspace $H^1_{q, L}(X)\cap L^q(X)$, has a unique extension to $H^1_{q, L}$.
    %    \item Each continuous linear functional $\ell$ on $H^1_{q, L}(X)$ can be represented by a function $g \in BMO_{q', L^*}(\real^d; X^*)$:
    %    \[
    %        \ell(f) = \int_{ \real^d } \lara{ f(x), g(x) }_{X \times X^*} \,\d x, \quad\forall\, f \in H^1_{q, L}(X)\cap L^q(X).
    %    \]
    %\end{enumerate}
\end{thm}

\begin{remark}
    Indeed, we can also obtain this duality theorem under the assumption that $L$ be a sectotrial operator satisfying only \eqref{pt1}, see Remark \ref{extend:PiL}.
\end{remark}

The more essential auxiliary result  is the following duality property, which is inspired by \cite[Theorem 2.4]{Mei2007} (see also \cite{Xia2016,Xu2022}).

\begin{prop}\label{GSB}
Let $X$ be any fixed Banach space and $1 \leq p < q$. Let $L$ be any fixed sectorial operator satisfying \eqref{pt10}, \eqref{pt2} and \eqref{pt3}. Then for any $f \in C_c(\real^d) \ot X$ and $g \in C_c(\real^d) \ot X^*$, one has
\begin{equation}\label{5.3}
    \jdz{ \int_{\real^d} \lara{f(x), g(x)}_{X \times X^*}dx }  \lesssim_{\gamma,\beta} \| G_{q,L}(f) \|_p^{\frac{p}{q}}\| S_{q,L}(f) \|_{p}^{1-\frac{p}{q}}\| g \|_{BMO^{p'}_{q', L^*}(X^*)}.
\end{equation}
\end{prop}

\begin{proof}   
Fixing $ f \in C_c(\real^d) \ot X $ and $ g \in C_c(\real^d) \ot X^* $, we consider truncated versions of $ G_{q,L}(f)(x) $ as follows:
\[
    G(x,t) := \sk{ \int_{t}^{\infty} \|\Q(f)(x, s)\|_X^q \frac{dxds}{s} }^{\frac{1}{q}},\quad x \in \real^d,t>0.
\]
By approximation, we can assume that $ G(x,t)$ is strictly positive. The operator $-tL^*e^{-tL^*}$ is denoted by $\Q^*$. By the Calder\'on identity---\eqref{funcal}, we have
\begin{align*}
   \jdz{ \int_{\real^d} \lara{f(x), g(x)}_{X \times X^*}\, \d x } & = 4 \jdz{\int_{\real^{d+1}_+} \lara{ \Q(f)(x, t), \Q^*(g)(x, t) }_{X \times X^*}\,\frac{\d x\d t}{t} } \\
   & = 4 \jdz{\int_{\real^{d+1}_+} \lara{ G^{\frac{p-q}{q}}(x,t)\Q(f)(x, t), G^{\frac{q-p}{q}}(x,t)\Q^*(g)(x, t) }_{X \times X^*}\,\frac{\d x\d t}{t}} \\
   & \lesssim \sk{ \int_{\real^{d+1}_+}G^{p-q}(x,t)\| \Q(f)(x, t) \|_X^q \,\frac{\d x\d t}{t} }^{\frac{1}{q}} \\
   & \qquad \cdot\sk{\int_{\real^{d+1}_+} G^{\frac{q-p}{q-1}}(x,t)\| \Q^*(g)(x, t) \|_{X^*}^{q'}\,\frac{\d x\d t}{t}}^{\frac{1}{q'}} \\
   & = I\cdot II.
\end{align*}

The term $I$ is estimated as below,
\begin{align*}
I^q & = - \int_{\real^d} \int_{0}^{\infty} G^{p-q}(x,t)\partial_t\sk{ G^q(x, t) }  \, \d t\d x \\
& = -q \int_{\real^d}\int_{0}^{\infty} G^{p-1}(x,t)\partial_t G(x,t) \, \d t\d x  \\
& \leq -q \int_{\real^d}\int_{0}^{\infty} G^{p-1}(x,0)\partial_t G(x,t) \, \d t\d x \\ 
& = q\int_{\real^d}G^p(x,0) \, \d t\d x  = q\| G_{q,L}(f) \|_p^p,
\end{align*}
since $G(x, t)$ is decreasing in $t$, and $G(x, 0) = G_{q,L}(f)(x)$.

For the term $II$, we introduce two more variants of $S_{q, \gamma,\beta}(f) $ (cf. \cite{Xu2022}). The first is defined similarly to $ G(\cdot,t) $:
\[
    S(x,t) = \sk{ \int_{t}^{\infty}\int_{|y-x| <s-\frac{t}{2} } \sk{ A_{\gamma,\beta}(f)(y,s) }^q \,\frac{\d y\d s}{s^{d+1}} }^{\frac{1}{q}},\quad x\in \real^d,\ t>0.
\]
To introduce the second one, let $ \D_k $ be the family of dyadic cubes in $\real^d$ of side length $ 2^{-k} $, that is, 
\[
    \D_k = \lk{ 2^{-k}\prod_{ j = 1 }^d[m_j, m_j+1): m_j \in \ent, k \in \ent }.
\]
Denote $ c_Q $ as the center of a cube $ Q $. Then, we define
\[
    \S(x,k) = \sk{ \int_{\sqrt{d}2^{-k}}^{\infty}  \int_{|y-c_Q|< s} \sk{ A_{\gamma,\beta}(f)(y,s) }^q \,\frac{\d y\d s}{s^{d+1}} }^{\frac{1}{q}}, \quad \text{if } x \in Q \in \D_k,\ k\in \ent .
\]
By definition, we have the following properties,
\begin{itemize}
    \item[(i)] $ \S(\cdot,k) $ is increasing in $ k $,
    \item[(ii)] $ \S(\cdot,k) $ is constant on every cube $ Q \in \D_k $,
    \item[(iii)] $ \S(x,-\infty) = 0 $ and $ \S(x,\infty) = S(x,0) = S_{q, \gamma,\beta}(f)(x) $.
\end{itemize}  
If $ s \geq t \geq \sqrt{d}2^{-k} $ and $ x\in Q \in \D_k $, then $ B(x,s-\frac{t}{2}) \subset B(c_Q,s) $, where $ B(x,t) $ denotes the ball with center $ x $ and radius $ t $. This implies
\[
    S(x,t) \leq \S(x,k), \quad  x \in Q \in \D_k \text{ whenever } t \geq \sqrt{d}2^{-k}.
\]
Using \eqref{SG2} and \eqref{Square Wilson} we have
\[
    G_{q,L}(f)(x) \lesssim G_{q, \gamma,\beta}(f)(x) \approx_{\gamma, \beta} S_{q, \gamma,\beta}(f)(x),
\]
and similarly,
\begin{equation}\label{G_xt-S_xt}
    G(x,t) \lesssim_{\gamma, \beta} S(x,t).
\end{equation}
Now we proceed to estimate the term $B$ based on these observations. {Applying \eqref{G_xt-S_xt} to $II$, we have}
\begin{align*}
  II^{q'} & \lesssim_{\gamma, \beta} \int_{\real^{d+1}_+} S^{\frac{q-p}{q-1}}(x,t)\| \Q^*(g)(x, t) \|_{X^*}^{q'}\,\frac{\d x\d t}{t} \\
   & = \sum_{k=-\infty}^{\infty} \sum_{Q \in \D_k} \int_Q\int_{\sqrt{d}2^{-k}}^{\sqrt{d}2^{-k+1}}S^{\frac{q-p}{q-1}}(x,t)\| \Q^*(g)(x, t) \|_{X^*}^{q'}\,\frac{\d t}{t}\d x \\
   & \leq \sum_{k=-\infty}^{\infty} \sum_{Q \in \D_k} \int_Q\int_{\sqrt{d}2^{-k}}^{\sqrt{d}2^{-k+1}}\S^{\frac{q-p}{q-1}}(x,k)\| \Q^*(g)(x, t) \|_{X^*}^{q'}\,\frac{\d t}{t}\d x \\
   & = \int_{\real^d} \sum_{k=-\infty}^{\infty} \sum_{j = -\infty}^{k} D(x,j)\int_{\sqrt{d}2^{-k}}^{\sqrt{d}2^{-k+1}}\| \Q^*(g)(x, t) \|_{X^*}^{q'}\,\frac{\d t}{t}\d x
\end{align*}
where $ D(x,j) = \S^{\frac{q-p}{q-1}}(x,j) - \S^{\frac{q-p}{q-1}}(x,j-1) $. Then $ D(x,j) $ is constant on every cube $ Q \in \D_j $. Thus
\begin{align*}
  II^{q'} & \lesssim_{\gamma,\beta} \int_{\real^d} \sum_{j=-\infty}^{\infty}D(x,j)\sk{ \sum_{k = j}^{\infty} \int_{\sqrt{d}2^{-k}}^{\sqrt{d}2^{-k+1}}\| \Q^*(g)(x, t) \|_{X^*}^{q'} \,\frac{\d t}{t} }\d x \\
   & = \sum_{j=-\infty}^{\infty} \sum_{Q \in \D_j} \int_Q D(x,j) \int_{0}^{\sqrt{d}2^{-j+1}}\| \Q^*(g)(x, t) \|_{X^*}^{q'}\,\frac{\d t}{t}\d x \\
   & = \sum_{j=-\infty}^{\infty} \sum_{Q\in\D_j}  D(x,j) \1_Q(x) \int_Q\int_{0}^{2\sqrt{d}\ell(Q)}\| \Q^*(g)(x, t) \|_{X^*}^{q'}\,\frac{\d t}{t}\d x,
\end{align*}
where $\ell(Q)$ denotes the length of $Q$. There exists a ball $ B $ such that $ Q \subset B$, $Q\times (0,2\sqrt{d}\ell(Q)] \subset \wh{B} $ and $  |B| \lesssim |Q| $. Then we deduce that
\[
    \int_Q\int_{0}^{2\sqrt{d}\ell(Q)}\| \Q^*(g)(x, t) \|_{X^*}^{q'}\,\frac{\d t}{t}\d x \leq \inf_{y\in B} \lk{\mathcal{C}_{q'}\mk{\Q^*(g)}(y)}^{q'}|B| \lesssim \inf_{y\in Q} \lk{\mathcal{C}_{q'}\mk{\Q^*(g)}(y)}^{q'}|Q|.
\]

Therefore
\begin{align*}
  II^{q'} & \lesssim_{\gamma,\beta} \sum_{j=-\infty}^{\infty} \sum_{Q\in\D_j}  D(x,j) \1_Q(x) \inf_{y\in Q} \lk{\mathcal{C}_{q'}\mk{\Q^*(g)}(y)}^{q'}|Q|\\
   & \leq \sum_{j=-\infty}^{\infty} \sum_{Q\in\D_j} \int_Q D(x,j) \sk{ \mathcal{C}_{q'}\mk{\Q^*(g) }(x) }^{q'}\,\d x \leq \int_{\real^d}\sum_{j=-\infty}^{\infty}D(x,j) \sk{ \mathcal{C}_{q'}\mk{\Q^*(g) }(x) }^{q'}\,\d x \\
   & = \int_{\real^d} \S^{\frac{q-p}{q-1}}(x,\infty)\sk{ \mathcal{C}_{q'}\mk{\Q^*(g) }(x) }^{q'}\,\d x = \| S^{\frac{q-p}{q-1}}_{q, \gamma,\beta}(f) \|_r \norm{ \sk{ \mathcal{C}_{q'}\mk{\Q^*(g) } }^{q'} }_{r'}  \\
   & = \| S_{q, \gamma,\beta}(f) \|_p^{\frac{q-p}{q-1}}\| \mathcal{C}_{q'}\mk{\Q^*(g)} \|_{p'}^{q'},
\end{align*}
where $1/r = 1 - q'/p' = (q-p)/(qp-p)$.

Combining the estimates of $I$ and $II$ with Theorem \ref{equiv}, we get the desired assertion.
\end{proof}

Finally, we arrive at the proof of our main theorem.

\begin{proof}[Proof of Theorm \ref{main2}]
    The first part \eqref{SL=SD} of Theorem \ref{main2} is a consequence of Theorem \ref{equiv}. Indeed, suppose $L$ be a generator such that the kernels of the generating semigroup satisfy \eqref{pt10} , \eqref{pt2} and \eqref{pt3} with $0 < \beta, \gamma \leq 1$, then the classical Poisson semigroup generated by $\sqrt{\Delta}$ satisfy obviously the same assumptions. Then
    \[ 
        \| S_{q,L}(f) \|_p \lesssim \| S_{q, \gamma, \beta}(f) \|_p \lesssim_{\gamma, \beta}  p^{\frac{1}{q}}\| S_{q, \sqrt{\Delta}}(f)  \|_p, \quad 1 \leq p < \infty.
    \]
    Similarly we obtain
    \[
        \| S_{q, \sqrt{\Delta}}(f)  \|_p \lesssim_{\gamma, \beta}p^{\frac{1}{q}}\| S_{q,L}(f) \|_p, \quad 1 \leq p < \infty.
    \]
    
    As for another part \eqref{SL=SG}, one side is easy by Theorem \ref{equiv},
    \[
        \| G_{q,L}(f) \|_p \lesssim \| G_{q, \gamma, \beta}(f) \|_p \approx_{\gamma, \beta} \| 
        S_{q, \gamma, \beta}(f) \|_p \lesssim_{\gamma, \beta} p^{\frac{1}{q}}\| S_{q,L}(f) \|_p, \quad 1 \leq p < \infty.
    \]   
    For the reverse direction, by Theorem \ref{Dual} and Proposition \ref{GSB}, we have for $1 \leq p < (1+q)/2$,
    \[
    \begin{aligned}
        \| f \|_{H^p_{q, L}(X) } & \lesssim \sk{ \frac{p(q-1)}{q-p} }^{\frac{1}{q'}}\sup_{ g } \jdz{ \int_{\real^d} \lara{f(x), g(x)}_{X \times X^*} \,\d x }\\
        & \lesssim_{\gamma, \beta} \sup_{ g } \| G_{q,L}(f) \|_p^{\frac{p}{q}}\| S_{q,L}(f) \|_p^{1- \frac{p}{q}} \| g \|_{BMO^{p'}_{q', L^*}(X^*)} \\
        & \lesssim_{\gamma, \beta} \| G_{q,L}(f) \|_p^{\frac{p}{q}}\| S_{q,L}(f) \|_p^{1- \frac{p}{q}},
    \end{aligned}
    \]
    where the supremum is taken over all $g\in C_c(\mathbb R^d)\otimes X^*$ such that  its $BMO_{q', L^*}^{p'}(X^*)$-norm is not more than $1$. Hence
    \[
        \| S_{q,L}(f) \|_p \lesssim_{\gamma, \beta} \| G_{q,L}(f) \|_p, \quad 1 \leq p < \frac{1+q}{2}.
    \]
    Now we deal with the case $(1+q)/2 \leq p < \infty$. Let $f\in C_c(\real^d)\otimes X$, we deduce from Theorem \ref{Dual} that
    \[
    \begin{aligned}
        \| f \|_{H^p_{q, L}(X)} & \lesssim \max\lk{ p'^{\frac{1}{q'}}p^{\frac{1}{q}}, p'^{\frac{2}{q'}} } \sup_{ h } \jdz{ \int_{ \real^d }\lara{f(x), h(x)}_{X \times X^*} \,\d x } \\
        & \lesssim p^{\frac{1}{q}}\sup_{ h } \jdz{ \int_{ \real^{d + 1}_+ }\lara{\Q(f)(x, t), \Q^*(h)(x, t)}_{X \times X^*} \,\frac{\d x\d t}{t} } \\
        & \lesssim p^{\frac{1}{q}} \sup_{ h } \| G_{q,L}(f) \|_p\| G_{q', L^*}(h) \|_{p'} \\
        & \lesssim_{\gamma, \beta} p^{\frac{1}{q}}p'^{\frac{1}{q'}} \sup_{ h } \| G_{q,L}(f) \|_p\| S_{q', L^*}(h) \|_{p'} \\
        & \lesssim p^{\frac{1}{q}}\| G_{q,L}(f) \|_p,
    \end{aligned}
    \]
    where the supremum is taken over all $h\in C_c(\real^d)\otimes X^*$ such that its $H^{p'}_{q', L^*}(X^*)$-norm is not more than $1$. 

    Combining the estimations above we conclude that 
    \[
        p^{-\frac{1}{q}}\| S_{q,L}(f) \|_p \lesssim_{\gamma, \beta} \| G_{q,L}(f) \|_p \lesssim_{\gamma, \beta} p^{\frac{1}{q}}\| S_{q,L}(f) \|_p, \quad  1 \leq p < \infty.
    \]
    We complete the proof.
\end{proof}

\section{Applications}\label{Applications}
In this section, we first recall the previous related results in \cite{Ouyang2010, MTX}. These, together with the tent space theory and Theorem \ref{main2}, will enable us to obtain the optimal Lusin type constants and the characterization of martingale type. In particular, this resolves partially Problem 1.8, Problem A.1 and Conjecture A.4 in the recent paper of Xu \cite{Xu2021}. 

Some notions and notations need to be presented.
%In the last part of our paper, we delve into the relationships and distinctions between our newly introduced Hardy spaces $H^p_{q, L}(\real^d; X)$ and the atomic vector-valued Hardy space $H^1_{\at}(\real^d; X)$ as well as vector-valued $L^p$-spaces. From these we obtain some opmital order of the constants in the vector-valued Littlewood-Paley-Stein inequality.
We first introduce the vector-valued atomic Hardy space $H^1_{\at}(\real^d; X)$. A measurable function $a\in L^{\infty}(\real^d; X)$ is called an $X$-valued atom if 
\[
    \supp\, a \subset B, \ \int_{\real^d} a(x) \,\d x = 0,\ \| a \|_{L^{\infty}(X)} \leq |B|^{-1},
\]
where $B$ is a ball in $\real^d$. The atomic Hardy space $H^1_{\at}(\real^d; X)$ is defined as the function space consisting of all functions $f$ which admits an expression of the form
\[
    f = \sum_{j =1}^{\infty}\lambda_ja_j, \quad \sum_{j =1}^{\infty} |\lambda_j| < \infty,
\]
where $a_j$ is an $X$-valued atom. The norm of $H^1_{\at}(\real^d; X)$ is defined as
\[  
    \| f \|_{H^1_{\at}(X)} = \inf\lk{ \sum_{j =1}^{\infty} |\lambda_j|: f(x) = \sum_{j =1}^{\infty} \lambda_ja_j(x) }.
\]
This is a Banach space.

The BMO space $BMO(\real^d; X)$ is defined as the space of all $f \in L^1_{{\rm loc}}(\real^d; X)$ equipped with the semi-norm
\[
    \| f \|_{BMO(X)} = \sup_{B} \frac{1}{|B|}\int_B \| f - f_B \|_X \,\d x < \infty,
\]
where the supremum runs over all the balls in $\real^d$ and $f_B$ represents the average of $f$ over $B$. $BMO(\real^d; X)$ is a Banach space modulo constants.

It is well-known that $BMO(\real^d; X^*)$ is isomorphically identified as a subspace of the dual space of $H^1_{\at}(\real^d; X)$ (cf. \cite{Bourgain1986}) and it is norming in the following sense 
\[
    \| f \|_{H^1_{\at}(X)} \approx \sup\lk{ |\lara{f, g}| : g \in BMO(\real^d; X^*),\ \| g \|_{BMO(X)}\leq 1 }
\]
with universal constants. Furthermore, if the Banach space $X^*$ has the Radon-Nikod{\'y}m property, then (cf. \cite{Blasco1988})
\begin{align}\label{dualatomic}
    (H^1_{\at}(\real^d; X))^* = BMO(\real^d; X^*),
\end{align}
with equivalent norms.

%As outlined in the reference \cite{Pi1975}, we define the modulus of convexity and modulus of smoothness of a Banach space $X$ as follows
%\[
%\begin{aligned}
%    \delta_X(\e) & = \inf\lk{ 1 - \norm{ \frac{a+b}{2} }_X: a, b \in X, \|a\|_X= \|b\|_X = 1, \| a-b \|_X = \e }, \quad 0 < \e < 2.\\
%    \rho_X(t) & = \sup\lk{ \frac{ \| a+tb \|_X + \| a-tb \|_X }{2}: a, b \in X, \| a \|_X = \| b \|_X = 1 }, \quad t>0.
%\end{aligned}
%\]
%The norm $\| \cdot \|_X$ is called $q$-uniform convex if $\delta_X(\e) \geq c\e^q $ for some positive constants $c$ and $q$; $\| \cdot \|_X$ is called $q$-uniform smooth if $ \rho_X(t) \leq ct^q $ for some positive constants $c$ and $q$. Readers are referred to \cite{Pi1975} for further details. For convenience, we call $X$ is $q$-uniform convex (respectively, smooth) if $X$ admits an equivalent $q$-uniform convex (respectively, smooth) norm.

We recall the following definitions on the geometric properties for Banach spaces. A Banach space $X$ is said to be of \emph{martingale type} $q$ (with $1<  q\leq2$) if there exists a positive constant $c$ such that every finite $X$-valued $L^q$-martingale $(f_n)_{n \geq 0}$, the following inequality holds
\[
   \sup_{n \geq 0} \ce\| f_n \|_X^q \leq c^q \sum_{ n \geq 1}\ce\| f_n - f_{n -1} \|_X^q,
\]
where $\ce$ denotes the underlying expectation; and the least constant $c$ is called the martingale type constant, denoted as $\mathsf{M}_{t, q}(X)$. While $X$ is said to be of \emph{martingale cotype} $q$ (with $2\leq q<\infty$) if the reverse inequalities holds with $c^{-1}$ in place of $c$ and the corresponding martingale cotype constant is denoted by $\mathsf{M}_{c, q}(X)$. Pisier's famous renorming theorem shows that $X$ is of martingale cotype (respectively, type) $q$ if and only if $X$ admits an equivalent $q$-uniform convex (respectively, smooth) norm. We refer the reader to \cite{Pi1975, Pi1986,Pisier2016martingales} for more details. 

Let $X$ be a Banach space and $1 < q \leq 2$. The authors in \cite{MTX} showed that the assertion that $X$ is of martingale type $q$ is equivalent to the one that for any $1<p<\infty$, there exists a constant $c_p$ such that for any $f \in C_c(\real^d)\otimes X$, 
        \begin{align}\label{MTXunoptimal}
            \| f \|_{L^p(X)} \leq c_p\| S_{q, \sqrt{\Delta}}(f) \|_p.
        \end{align}
        Later on, in \cite{Ouyang2010} the authors investigated the relationships between $H^1_{\at}(\real^d; X)$ and $H^1_{q, \sqrt{\Delta}}(\real^d; X)$ as well as the ones between $BMO(\real^d; X)$ and $BMO_{q, \sqrt{\Delta}}(\real^d; X)$, and provided insights into the geometric properties of the underlying Banach space $X$.

\begin{thm}\label{thm:ouxu}
    Let $X$ be a Banach space and $1 < q \leq 2$. The followings are equivalent
    \begin{enumerate}[{\rm (i)}]
        \item $X$ is of martingale type $q$; 
        \item there exists a positive constant c such that for any $f \in C_c(\real^d)\otimes X$, 
        \[
            \| f \|_{H^1_\at(X)} \leq c\| S_{q, \sqrt{\Delta}}(f) \|_{1};
        \]
        \item there exists a positive constant c such that for any $f \in C_c(\real^d)\otimes X$, 
        \[
            \| f \|_{BMO(X)} \leq c\| f \|_{BMO_{q, \sqrt{\Delta}}(X)}.
        \]
    \end{enumerate}
    Moreover, the constants in {\rm (ii)} and {\rm (iii)} are majored by $\mathsf{M}_{t, q}(X)$.
   % Here $ \mathbf{F} $ is the projection from $L^p(\real^d; X)$ onto the fixed point space of $\lk{e^{-t\sqrt{\Delta}}}_{t > 0}$.  
    %The reverse inequalities hold if $X$ is of martingale cotype $q$ ($2 \leq q < \infty$), with adjustments to the constants involved. 
\end{thm}

The following theorem follows from the interpolation theory between vector-valued tent spaces---Lemma \ref{Intpo}---and the boundedness of the projection $\pi_L$---Lemma \ref{l3.8}. See for instance the general interpolation theory of complemented subspaces (cf. \cite[Section 1.17]{Triebel1978}), and we omit the details.
 \begin{thm}\label{thm:interHL}
  Let $X$ be any fixed Banach space, $1<q<\infty$ and $1 \leq p_1 < p < p_2 < \infty$ such that $1/p = (1-\theta)/p_1 + \theta/p_2$ with $0\leq\theta\leq1$. Let $L$ be as in Theorem \ref{main2}.
  Then
   \[
        [H^{p_1}_{q, L}(\real^d; X), H^{p_2}_{q, L}(\real^d;X)]_\theta = H^p_{q, L}(\real^d;X),
   \]
   with equivalent norms, where $[\cdot, \cdot]_\theta$ is the complex interpolation space. More precisely, for $f\in C_c(\real^d)\otimes X$, one has
   \[
        \|f \|_{H^p_{q, L}(X)} \lesssim \| f \|_{[H^{p_1}_{q, L}(\real^d; X), H^{p_2}_{q, L}(\real^d;X)]_\theta} \lesssim p^{\frac{2}{q}} \| f \|_{H^p_{q, L}(X)}.
   \]
\end{thm}

%{\color{blue} This assertion can put in Section 3?

%The order of the corresponding constant is not given in \cite{Ouyang2010}. We obtain it via interpolation between vector-valued Hardy spaces.

%\begin{proof}
%    Recall that the operator $\Q : H^p_{q, L}(\real^d; X) \to T^p_q(\real^{d+1}_+; X) $ is an immersion with its range $\Omega_{p, L}$. Moreover, the operator $\mathcal{P}$ mentioned in Remark \ref{Proj} is a continuous projection from $T^p_q(\real^{d+1}_+; X)$ onto itself with range $\Omega_{p, L}$. Then by the general interpolation theory of complemented subspaces (cf. \cite[Section 1.17]{Triebel1978}), the corollary follows the interpolation theory between vector-valued tent spaces---Lemma \ref{Intpo}. 
%\end{proof}

Now we are at the position to give the applications.

\begin{corollary}\label{cor:strong}
    Let $X$ be a Banach space and $1 < q \leq 2$. Let $L$ be as in Theorem \ref{main2}. The followings are equivalent
    \begin{enumerate}[{\rm (i)}]
        \item $X$ is of martingale type $q$;
%        \item 
%        \[
%            \| f \|_{L^p(\real^d) } \lesssim_{p, d, \gamma, \beta} \| G_{q, L}(f) \|_p, \quad 1< p < \infty
%        \]
%        for all continuous functions $f: \real^d \to X$  with compact support.
        \item for any $f \in C_c(\real^d)\otimes X$, 
        \[
            \| f \|_{H^1_{\at}(X) } \lesssim_{\gamma, \beta} \mathsf{M}_{t, q}(X)\| G_{q, L}(f) \|_1;
        \]
        \item for any $1<p<\infty$ and $f \in C_c(\real^d)\otimes X$,
        \[
            \| f  \|_{L^p(X)} \lesssim_{\gamma, \beta} p \mathsf{M}_{t, q}(X)\| G_{q, L}(f) \|_p;
        \]
        \item for any $f \in C_c(\real^d)\otimes X$,
        \[
            \| f \|_{BMO(X)} \lesssim_{\gamma, \beta} \mathsf{M}_{t, q}(X)\| f \|_{BMO_{q, L}(X)}.
        \]
    \end{enumerate}
    %The Littlewood-Paley $g$-function can be replaced by Lusin area integral, too. The corresponding best constant in {\rm (iii)} is denoted by $\mathsf{L}_{t, p, q}^L$. In particular, we obtain
\end{corollary}

\begin{proof}
    %{\color{red}PLEASE give the proof or brief explanation for each implication, and take care of the constants as precisely as possible!!! The proof of (iii) is form Theorem \ref{main2}. The estimation of constant need Theorem \ref{thm:interHL}.}
(i)$\Leftrightarrow$(ii). This follows immediately from Theorem \ref{main2} and Theorem \ref{thm:ouxu}.

(iii)$\Rightarrow$(i). This is deduced from Theorem \ref{main2} and \eqref{MTXunoptimal}.

(i)$\Rightarrow$(iii). In the case $1 < p < q$, by Theorem \ref{main2}, it suffices to show that
\begin{align}\label{1pq}
    \| f \|_{L^p(X)} \lesssim \mathsf{M}_{t, q}(X)\| S_{q, \sqrt{\Delta}} \|_p.
\end{align}
Keeping in mind \eqref{MTXunoptimal} and Theorem \ref{thm:ouxu} (ii), we consider
\[
    [H^1_{q, \sqrt{\Delta}}(\real^d; X), H^{q}_{q, \sqrt{\Delta}}(\real^d; X)]_\theta \subset [H^1_{\at}(\real^d; X), L^q(\real^d; X)]_{\theta};
\]
then combining Theorem \ref{thm:interHL} with the interpolation between $H^1_{\at}(\real^d; X)$ and $L^q(\real^d; X)$ (cf. \cite[Theorem A]{Blasco1989}), one gets for any $f \in C_c(\real^d)\ot X$, $1/p = 1-\theta + \theta/q$,
\[
    \| f \|_{L^p(X)} \lesssim \| f \|_{ [H^1_{\at}(\real^d; X), L^q(\real^d; X)]_{\theta} } \lesssim \mathsf{M}_{t, q}(X)\| f \|_{[H^1_{q, \sqrt{\Delta}}(\real^d; X), H^{q}_{q, \sqrt{\Delta}}(\real^d;X)]_\theta} \lesssim \mathsf{M}_{t, q}(X)\| f \|_{H^p_{q, \sqrt{\Delta}}(X)}.
\]
This is the desired \eqref{1pq}. 
%Then from Theorem \ref{main2} we observe
%\[
%    \| f \|_{L^p(X)} \lesssim \| S_{q, \sqrt{\Delta}}(f) \|_p \approx_{\gamma, \beta} \| G_{q, L}(f) \|_p, \quad 1 < p < q.
%\]
Combining it with the related result for $q \leq p < \infty$ in \cite{Xu2021}, we conclude 
\[
    \| f \|_{L^p(X)} \lesssim_{\gamma, \beta} p\mathsf{M}_{t, q}(X)\| G_{q, L}(f) \|_p, \quad 1 < p < \infty .
\]

(i)$\Leftrightarrow$(iv). This follows from Remark \ref{equiv-BMO} and Theorem \ref{thm:ouxu} (iii).
\end{proof}

\begin{remark}\label{SovlePro}
(1). Taking $L=\sqrt{\Delta}$  in the assertion (iii), we get {
    \[
        \mathsf{L}_{t, q, p}^{\sqrt{\Delta}} (X)\lesssim p\mathsf{M}_{t,q}(X),\quad 1 < p < \infty,
    \]}
where the order is optimal as $p$ tends to $1$. This solves partially \cite[Problem 1.8]{Xu2021}. 

(2). The implication (iii)$\rightarrow$(i) says that a Banach space $X$ which is Lusin type $q$ relative to $\{e^{-tL}\}_{t>0}$ implies the martingale type $q$ for a large class of generators $L$. This answers partially \cite[Problem A.1 and Conjecture A.4]{Xu2021}.

\end{remark}

%\begin{remark}
 %   Recall that the complex plane $\com$ is of both martingale cotype $2$ and martingale type $2$. Hence when we return to the scalar-valued case, we conclude that for $1 < p < \infty$,
  %  \[
   % H^1_{\at}(\real^d) = H^1_{2, L}(\real^d),\ L^p(\real^d) = H^p_{2, L}(\real^d),\ BMO(\real^d) = BMO_{2, L}(\real^d) 
   % \]
    %with equivalent norms. The equivalence gives a characterization of Hardy space by Lusin area integral associated with certain sectotrial operators. Moreover, it can also be characterized by Littlewood-Paley $g$-functions.
%\end{remark}

\section{Appendix}
From the atomic decomposition of $T^1_q(\real^{d+1}_+; X)$, we derive the following molecular decomposition for $H^1_{q, L}(\real^d; X)$ for any Banach space $X$, which might have further applications.

\begin{thm}\label{atomdec}
Let $X$ be any fixed Banach space and $1<q<\infty$. For any $f \in H^1_{q, L}(\real^d; X)$, there exist a sequence of complex numbers $\lk{ 
\lambda_j }_{j \geq 1}$ and corresponding molecules $\alpha_j = \pi_L(a_j)$ with $a_j(x, t)$ being an $(X, q)$-atom such that
\[
    f = \sum_{j \geq 1}\lambda_j \alpha_j, \quad \|f\|_{H^1_{q, L}(X)} \approx \sum_{j \geq 1} |\lambda_j|.
\]
\end{thm}
\begin{proof}
    Let $f \in H^1_{q, L}(\real^d; X)$. It follows that $\Q(f) \in T^1_q(\real^{d+1}_+; X)$. Hence $\Q(f)$ admits an atomic decomposition by Lemma \ref{atmicDectent}. More precisely, there exist a sequence of complex numbers $\{ c_j \}_{j \geq 1}$ and $(X, q)$-atoms $a_j$ such that
    \[
        \Q(f) = \sum_{j =1}^{\infty} c_j a_j, \quad \| f \|_{H^1_{q, L}(X)} = \| \Q(f) \|_{T^1_q(X)} \approx \sum_{j = 1}^{\infty} |c_j|.
    \]
    Then by Lemma \ref{l3.8}, it follows that $\pi_L(a_j) =\alpha_j \in H^1_{q, L}(\real^d; X)$ for all $j \geq 1$. Recall below the Calder\'on identity---\eqref{funcal}, 
        \[
        f(x) = 4\int_0^{\infty} \Q[\Q(f)(\cdot, t)](x, t) \,\frac{\d t}{t}.
    \]
This further deduce that
    \[
        f(x) =  4\sum_{j =1}^{\infty} c_j \int_0^{\infty} \Q[a_j(\cdot, t)](x, t))\, \frac{\d t}{t} = 4\sum_{j =1}^{\infty} c_j\alpha_j(x),
    \]
    %The integral converges for almost every $x \in \real^d$. Let $\lambda_j = 4c_j$, 
    and thus we obtain the desired molecular decomposition. 
\end{proof}

\bigskip {\textbf{Acknowledgments.}}The subject of this article was proposed by Professor Quanhua Xu when he was completing his paper \cite{Xu2021}. The authors are very grateful to his suggestions, encouragement, support as well as his careful reading over the whole paper. We also would like to thank Professor \'{E}ric Ricard for his comments and suggestions.
G. Hong was supported by National Natural Science Foundation of China (No. 12071355, No. 12325105, No. 12031004) and Fundamental Research Funds for the Central Universities (No. 2042022kf1185).

%%%%%%%%%%%%%%%%%%%%%%%%%%%%%%%%%%%%%%%%%%%%%%%%%%%%%%
%%%%%%%%%%%%%%%%%%%%%%%%%%%%%%%%%%%%%%%%%%%%%%%%%%%%%%

\end{document}